\documentclass[a4paper,11pt,fleqn]{elsarticle}

\makeatletter
\def\ps@pprintTitle{%
  \let\@oddhead\@empty
  \let\@evenhead\@empty
  \def\@oddfoot{\reset@font\hfil\thepage\hfil}
  \let\@evenfoot\@oddfoot
}
\makeatother

\usepackage{amsmath, amsthm, amssymb, graphicx}
\usepackage{wasysym}
\usepackage[usenames,dvipsnames]{color}
\usepackage{subfig}

\usepackage{amsrefs}

\allowdisplaybreaks

\def\PS{\includegraphics[scale=0.3, bb = 5 8 39 33]{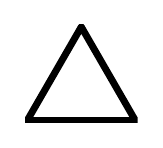}}
\def\PSA{\includegraphics[scale=0.3, bb = 5 8 39 33]{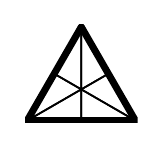}}
\def\PSB{\includegraphics[scale=0.3, bb = 5 8 39 33]{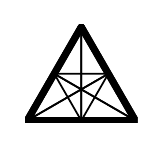}}
\def\PSC{\includegraphics[scale=0.3, bb = 5 8 39 33]{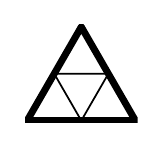}}

\def\RR{\mathbb{R}}

\def\mm{\mathbf{m}}
\def\nn{\mathbf{n}}
\def\tt{\mathbf{t}}
\def\uu{\mathbf{u}}
\def\vv{\mathbf{v}}
\def\ww{\mathbf{w}}

\def\BBB{\mathcal{B}}
\def\SSS{\mathcal{S}}
\newcommand{\norm}[1]{\lVert#1\rVert}

\def\ud{\text{d}}
\def\Sage{\texttt{Sage}}
\def\ip#1#2{\langle #1, #2\rangle}

\newtheorem{theorem}{Theorem}
\theoremstyle{definition}
\newtheorem{example}{Example}
\theoremstyle{remark}
\newtheorem{remark}[theorem]{Remark}

\begin{document}

\title{A Hermite interpolatory subdivision scheme\\ for $C^2$-quintics on the Powell-Sabin 12-split}

\author[a]{Tom Lyche}
\ead{tom@ifi.uio.no}
\author[b]{Georg Muntingh}
\ead{georg.muntingh@gmail.com}

\address[a]{Department of Mathematics, University of Oslo, PO Box 1053, Blindern, 0316 Oslo, Norge}
\address[b]{SINTEF ICT, PO Box 124 Blindern, 0314 Oslo, Norge}

\date{\today}

\begin{abstract}
In order to construct a $C^1$-quadratic spline over an arbitrary triangulation, one can split each triangle into 12 subtriangles, resulting in a finer triangulation known as the Powell-Sabin 12-split. It has been shown previously that the corresponding spline surface can be plotted quickly by means of a Hermite subdivision scheme \cite{DynLyche98}. In this paper we introduce a nodal macro-element on the 12-split for the space of quintic splines that are locally $C^3$ and globally $C^2$. For quickly evaluating any such spline, a Hermite subdivision scheme is derived, implemented, and tested in the computer algebra system \Sage. Using the available first derivatives for Phong shading, visually appealing plots can be generated after just a couple of refinements.
\end{abstract}

\maketitle

\section{Introduction}
\noindent For approximating functions on a given domain, a popular method is to triangulate the domain and consider an approximation in a space $\SSS$ of piece-wise polynomials over the triangulation. It is a hard problem to find a basis $\BBB$ of $\SSS$ that has all the usual properties of the univariate B-splines.

One desired property of $\BBB$ is that it is \emph{local}, meaning that each spline in $\BBB$ has local support. One way to construct such a local basis is to first split each triangle into several subtriangles, and then construct a basis on the refined triangulation.

A popular split is the Powell-Sabin 12-split \cite{Powell.Sabin77}; see Figure \ref{fig:PS12-Labels12} and Section~\ref{sec:macroelement}. While the 12-split splits the triangle in a relatively large number of subtriangles, a major advantage over other well-known splits stems from the following property \cite{Oswald92, DynLyche98}. Let be given a triangle $\PS$, its 12-split $\PSB$, and the split $\PSC$, where we subdivided $\PS$ into four subtriangles by connecting the midpoints of the edges. If we replace each subtriangle in $\PSC$ by its 12-split, the space of splines over the resulting split contains the space of splines over $\PSB$. This refinability property makes the 12-split suitable for multiresolution analysis.

Recently, a simplex spline basis for the $C^1$-quadratics on the 12-split with all the usual properties of the univariate B-spline basis was discovered \cite{Cohen.Lyche.Riesenfeld13}. Powell and Sabin originally constructed a nodal basis (see Section \ref{sec:BernsteinBezier}) on the 12-split, which can be used to represent $C^1$-smooth quadratic splines over arbitrary triangulations. Schumaker and Sorokina viewed the space of $C^1$-quadratics on the 12-split as the first entry in a sequence of spline spaces of increasing smoothness and degree \cite{Schumaker.Sorokina06}. The second entry is a space of $C^2$-quintics, with $C^3$-supersmoothness at the vertices and midpoints and satisfying some additional $C^3$-conditions of type \eqref{eq:smoothnessconditions} along some of the interior edges. On a single triangle this space has dimension $42$, and the authors constructed a nodal macro-element for this space; see Figure \ref{fig:PS12-0}.

For a recent and similar construction see \cite{Davydov.Yeo13}. A family of smooth spline spaces on the 6-split and corresponding normalized bases were presented in \cite{Speleers13}. For other refinable $C^1$-quadratic elements on 6-splits see \cite{D.L.M.S:S00,Maes.Bultheel06} and \cite{Jia.Liu08}. In the latter a combination of 6- and 12-splits is used. For the FVS $C^1$-cubic quadrangular macro element see \cite{Davydov.Stevenson05,Hong.Schumaker04}, and for a survey of refinable multivariate spline functions see~\cite{Goodman.Hardin06}. 

The next section reviews some standard Bernstein-B\'ezier techniques. Section \ref{sec:macroelement} introduces a new macro-element space for $C^2$-quintics, with complete $C^3$-smoothness within each macrotriangle and dimension only $39$; see Figure~\ref{fig:PS12-1}. In the following two sections a Hermite subdivision scheme is derived, implemented, and tested in the computer algebra system \Sage. The final section concludes the paper.

\begin{figure}
\subfloat[]{\includegraphics[scale=0.64, bb = 30 0 197 200]{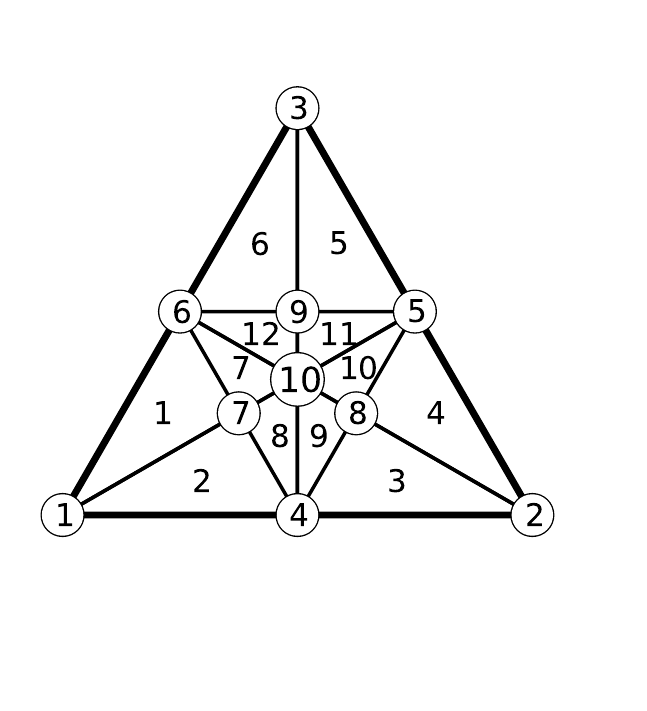}\label{fig:PS12-Labels12}}
\subfloat[]{\includegraphics[scale=0.64, bb = 10 0 197 200]{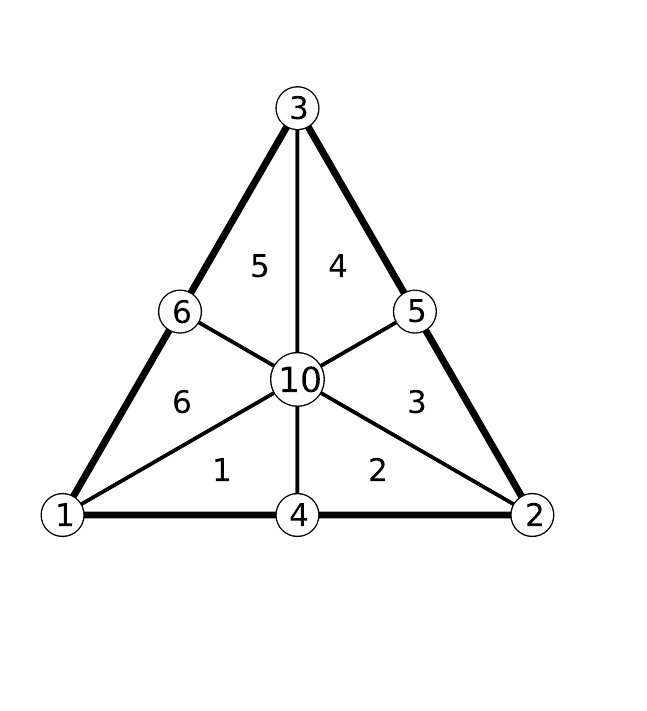}\label{fig:PS12-Labels6}}
\subfloat[]{\includegraphics[scale=0.64, bb = 20 0 180 200]{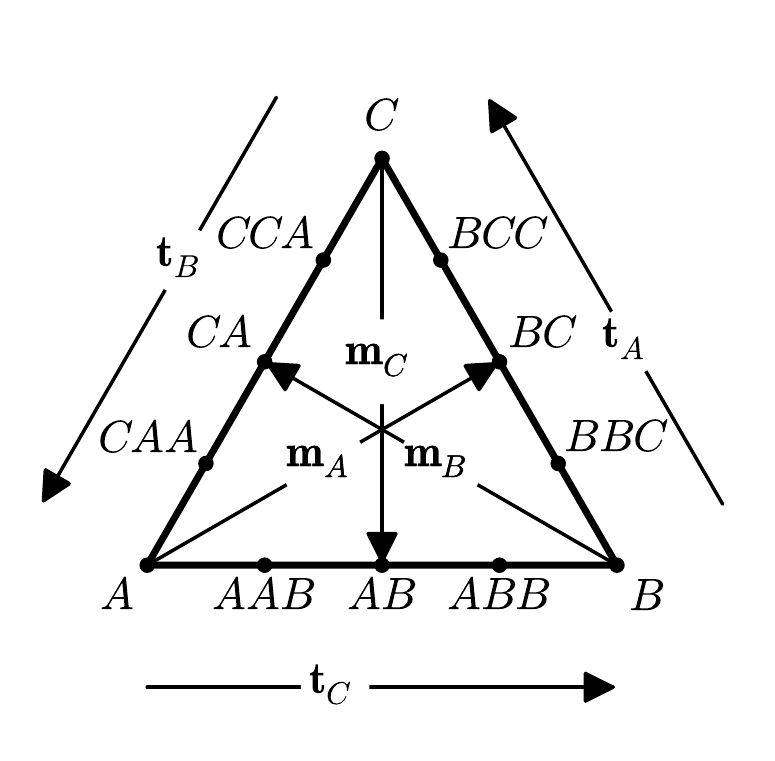}\label{fig:PS12-LabelsVectors}}
\caption{The Powell-Sabin 12-split (left) and 6-split (middle) with labeling of vertices and faces. Right: A triangle with corners $\vv_A, \vv_B, \vv_C$, midpoints $\vv_{AB}, \vv_{BC}, \vv_{CA}$, quarterpoints $\vv_{AAB}, \vv_{ABB}, \vv_{BBC}, \vv_{BCC}, \vv_{CCA}, \vv_{CAA}$, medial vectors $\mm_A, \mm_B, \mm_C$, and tangential vectors $\tt_A, \tt_B, \tt_C$.}\label{fig:SidesAndMedians}
\end{figure}

\begin{figure}
\begin{center}
\subfloat[]{\includegraphics[scale=0.69]{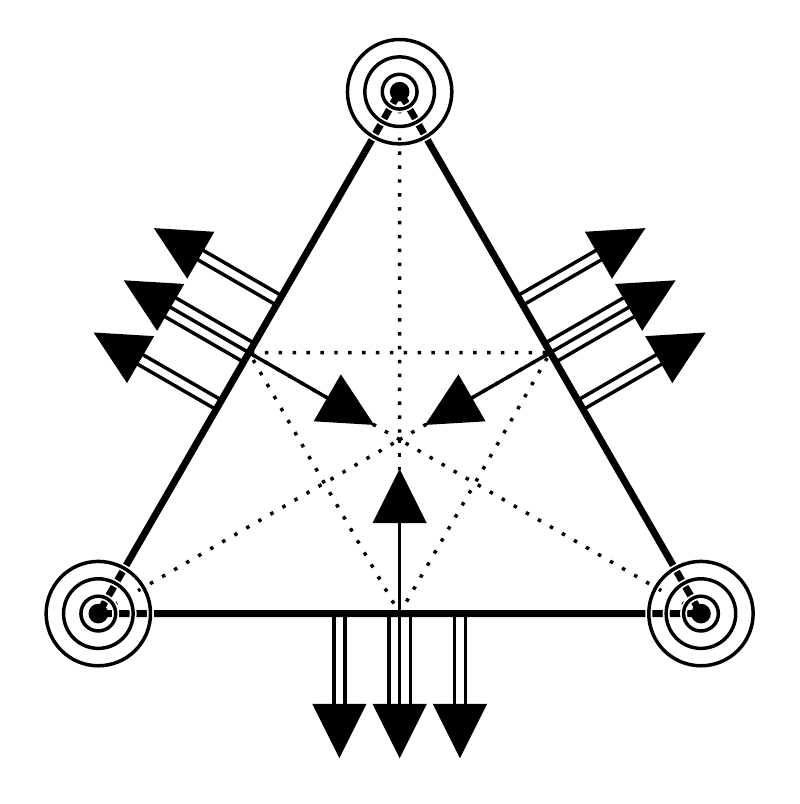}\label{fig:PS12-0}}\qquad 
\subfloat[]{\includegraphics[scale=0.69]{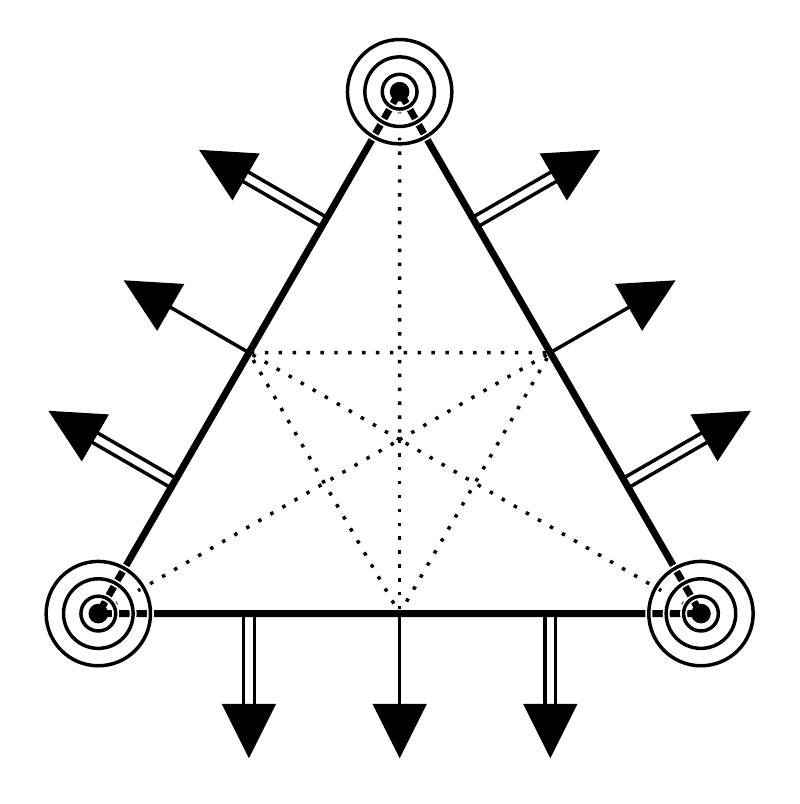}\label{fig:PS12-1}}
\end{center}
\caption{Schumaker and Sorokina's macro-element (left) and a new macro-element (right) on the 12-split. A bullet represents a point evaluation, three circles represent all derivatives up to order three, and a single, double, and triple arrow represent a first-, second-, and third-order directional derivative. These derivatives are evaluated at the rear end of the arrows, which are located at midpoints and adjacent domain points (left), and at the midpoints and quarterpoints (right).} \label{fig:PS12-MacroElements}
\end{figure}

\section{Bernstein-B\'ezier techniques}\label{sec:BernsteinBezier}
\noindent We follow the notation from \cite{Lai.Schumaker07}. Any point $\vv$ in a nondegenerate triangle $T = \langle \vv_1, \vv_2, \vv_3 \rangle$ can be represented by its \emph{barycentric coordinates} $(b_1,b_2,b_3)$, which are uniquely defined by $\vv = b_1 \vv_1 + b_2 \vv_2 + b_3 \vv_3$ and $b_1 + b_2 + b_3 = 1$.
Similarly, each vector $\uu$ is uniquely described by its \emph{directional coordinates}, i.e., the triple $(b_1 - b_1', b_2 - b_2', b_3 - b_3')$ with $(b_1, b_2, b_3)$ and $(b_1', b_2', b_3')$ the barycentric coordinates of two points $\vv$ and $\vv'$ such that $\uu = \vv - \vv'$.

A polynomial $p$ of degree $d$ defined on $T$ is conveniently represented by its \emph{B\'ezier form}
\[
p(\vv) = \sum_{i+j+k=d} c_{ijk} B_{ijk}^d(\vv), \qquad
B_{ijk}^d(\vv) := \frac{d!}{i!j!k!}b_1^i b_2^j b_3^k,
\]
where the $B_{ijk}^d$ are referred to as the \emph{Bernstein basis polynomials} of degree $d$ and the $c_{ijk}$ are called the \emph{B-coefficients} of $p$. We associate each B-coefficient $c_{ijk}$ to the \emph{domain point} $\xi_{ijk} := \frac{i}{d}\vv_1 + \frac{j}{d}\vv_2 + \frac{k}{d}\vv_3$. The \emph{disk of radius $m$ around $\vv_1$} is $D_m (\vv_1) := \{\xi_{ijk} : i \geq d - m \}$, and similarly for the other vertices.

For any differentiable function $f: \Omega \longrightarrow \RR$ and a vector $\uu\in \RR^2$ (not necessarily of unit length), we write
\[ f^\uu_\vv = \nabla_\uu f (\vv) := \left.\frac{\ud}{\ud t}f(\vv + t\uu)\right|_{t=0} \]
for the \emph{directional derivative} of $f$ along $\uu$ at $\vv$. If, more generally, $f$ is $k$-times differentiable for some integer $k\geq 1$, then, for any set of vectors $\{\uu_1, \uu_2, \ldots, \uu_k\}\subset \RR^2$, we write
\[ f^{\uu_1 \uu_2\cdots \uu_k} = \nabla_{\uu_1\uu_2\cdots\uu_k}f := \nabla_{\uu_1} \nabla_{\uu_2} \cdots \nabla_{\uu_k}f \]
for the \emph{$k$-th order directional derivative} in the directions $\uu_1, \uu_2,\ldots, \uu_k$ of $f$, and $f^{\uu_1 \uu_2\cdots \uu_k}_\vv$ for its point evaluation at $\vv$. We write
\[ \nabla^n_\uu f:=\nabla_{\uu_1\uu_2\cdots\uu_k}f,\]
if $ \uu=\uu_1=\cdots=\uu_k$.
We shall make use of the short-hand
\begin{equation}\label{eq:shorthand}
f^{\uu_1 \uu_2\cdots \uu_k}_{\vv \pm \ww} := f^{\uu_1 \uu_2\cdots \uu_k}_\vv \pm f^{\uu_1 \uu_2\cdots \uu_k}_\ww.
\end{equation}

For any triangle $\langle \vv_1, \vv_2, \vv_3\rangle$ there are some natural directions along which to consider directional derivatives. To any edge $\langle \vv_Y, \vv_Z \rangle$ opposing the vertex $\vv_X$, one associates the inward unit normal vector $\nn_X$, tangential vector $\tt_X := \vv_Z - \vv_Y$ and the medial vector $\mm_X := \frac12 (\vv_Y + \vv_Z) - \vv_X$; see Figure~\ref{fig:PS12-LabelsVectors}.
Since $\nn_X$ and $\tt_X$ are orthogonal,
\begin{equation}\label{eq:m-vs.-n}
\mm_X = \alpha_X\nn_X + \beta_X\tt_X,\qquad \alpha_X := \frac{\ip{\mm_X}{\nn_X}}{\ip{\nn_X}{\nn_X}},\qquad \beta_X := \frac{\ip{\mm_X}{\tt_X}}{\ip{\tt_X}{\tt_X}}.
\end{equation}
The corresponding directional derivatives obey similar relations.

A linear functional $\lambda : \SSS \longrightarrow \RR$ is called \emph{nodal} if it maps any spline to a linear combination of values and derivatives at a given point called the \emph{carrier} of $\lambda$. A set $\Lambda$ of nodal functionals is called a \emph{nodal determining set} for a spline space $\SSS$, if $\lambda s = 0$ for all $\lambda\in \Lambda$ implies $s\equiv 0$. If there is no smaller nodal determining set, $\Lambda$ is called \emph{minimal}. We call $\SSS$ a \emph{macro-element space} provided that there is a nodal minimal determining set $\Lambda$ for $\SSS$ such that for each triangle $T$ in the triangulation $\PS$, $s|_T$ is uniquely determined from the values $\{\lambda s \}_{\lambda \in \Lambda_T}$, with $\Lambda_T \subset \Lambda$ the functionals with carrier in $T$. In this case $\Lambda$ is a basis for the dual space of $\SSS$, and the dual basis to $\Lambda$ is called a \emph{nodal basis} of $\SSS$. The functionals in $\Lambda$ are called \emph{degrees of freedom}.

Suppose that $T := \langle \vv_1, \vv_2, \vv_3\rangle$  and $\tilde{T} := \langle \tilde{\vv}_1, \vv_2, \vv_3\rangle$ are triangles sharing the edge $e := \langle \vv_2, \vv_3\rangle$, and let
\[ p(\vv) := \sum_{i+j+k=d} c_{ijk} B_{ijk}^d (\vv), \qquad
\tilde{p}(\vv) := \sum_{i+j+k=d} \tilde{c}_{ijk} \tilde{B}_{ijk}^d (\vv)\]
be polynomials defined on these triangles, where $\{ B_{ijk} \}$ and $\{ \tilde{B}_{ijk} \}$ are the Bernstein basis polynomials associated with $T$ and $\tilde{T}$, respectively. Imposing a smooth join of $p$ and $\tilde{p}$ along $e$ translates into the linear relations among the B-coefficients presented in the following theorem.
\begin{theorem} 
Suppose $\uu$ is any direction not parallel to $e$. Then
\begin{equation}\label{eq:smoothnessconditions0}
\nabla^n_\uu\, p(\vv) = \nabla^n_\uu\, \tilde{p}(\vv)\qquad
\text{for all } \vv\in e\quad \text{ and }\quad n = 0,\ldots,r,
\end{equation} 
if and only if
\begin{equation}\label{eq:smoothnessconditions}
\tilde{c}_{njk} = \sum_{\nu + \mu + \kappa = n} c_{\nu,j+\mu,k+\kappa} 
\frac{n!}{\nu!\mu!\kappa!}b_1^\nu b_2^\mu b_3^\kappa,
\quad j+k=d-n,\quad n = 0,\ldots, r,
\end{equation}
where $b_1,b_2,b_3$ are the barycentric coordinates of $\tilde{\vv}_1$ with respect to $T$.
Moreover, if $\vv_1,\vv_3,\tilde{\vv}_1$ are collinear then \eqref{eq:smoothnessconditions} takes the form 
\begin{equation}\label{eq:smoothnessconditions1}
\tilde{c}_{njk} = \sum_{\nu = 0}^n c_{\nu,j,k+n-\nu} \binom{n}{\nu}b_1^ \nu(1-b_1)^{n-\nu},\qquad b_1=-\frac{\norm{\tilde{\vv}_1-\vv_3}}{\norm{\vv_1-\vv_3}},
\end{equation}
which can also be written
\begin{equation}\label{eq:smoothnessconditions2} 
\frac{\Delta^n e_0}{\norm{\vv_1-\vv_3}^n}=\frac{\Delta^n \tilde{e}_0}{\norm{\tilde{\vv}_1-\vv_3}^n},\end{equation}
where $\Delta ^n x_i=\sum_{j=0}^n(-1)^{n-j}\binom{n}{j}x_{i+j}$ is the usual $n$th order forward difference of $x_i,\ldots,x_{i+n}$, $ e_\nu=c_{\nu,j,k+n-\nu}$, and
$\tilde{e}_\nu=\tilde{c}_{n-\nu,j,k+\nu}$, $ \nu=0,1,\ldots,n$.
\end{theorem}

\begin{proof}
The equivalence of \eqref{eq:smoothnessconditions0} and \eqref{eq:smoothnessconditions} follows from \cite[Thm. 2.28]{Lai.Schumaker07} by switching the order of the vertices $\vv_2$ and $\vv_3$ in $\tilde T$. If $\vv_1,\vv_3,\tilde{\vv}_1$ are collinear, then $b_2=0$ and all terms with $\mu>0$ in \eqref{eq:smoothnessconditions} are zero. We obtain \eqref{eq:smoothnessconditions1} with $b_3=1-b_1$. It can be shown by induction on $n$ that \eqref{eq:smoothnessconditions2}  follows  from \eqref{eq:smoothnessconditions1}. Alternatively it follows from univariate Bernstein-B\'ezier theory.
\end{proof}

Let us illustrate this theorem with several examples that will be used in the proof of Theorem~\ref{thm:Macro-Element12}.

\begin{example}\label{ex:bivariateC1C2C3}
Let $T = \langle \vv_7, \vv_{10}, \vv_6 \rangle$ and $\tilde{T} = \langle \vv_9, \vv_6, \vv_{10} \rangle$ as in Figure \ref{fig:PS12-Labels12}. Then $\vv_9$ has barycentric coordinates $(-1, \frac32, \frac12)$ with respect to $T$, and the coefficients in \eqref{eq:smoothnessconditions} are, for $n = 1, 2, 3$ and $j = 0$, shown in Figure~\ref{fig:bivariateC1}--\ref{fig:bivariateC3}.
\end{example}

\begin{example}
See Figure~\ref{fig:univariate1}, \ref{fig:univariate2} for the coefficients in univariate $C^1$-, $C^2$-, and $C^3$-conditions across the edges $\langle\vv_1,\vv_7\rangle$, $\langle\vv_6,\vv_7\rangle$, and see Figure~\ref{fig:univariate3} for the coefficients of a univariate $C^3$-condition along the line segment $\langle\vv_4,\vv_9\rangle$. We can use either \eqref{eq:smoothnessconditions1} or \eqref{eq:smoothnessconditions2} to show this. We use \eqref{eq:smoothnessconditions2} with $\norm{\tilde{\vv}_1-\vv_3}=\norm{\vv_1-\vv_3}$ for (d) and $\norm{\tilde{\vv}_1-\vv_3}=\norm{\vv_1-\vv_3}/3$ for (f). For (e) we use \eqref{eq:smoothnessconditions1} with $b_1=-1/3$. 
\end{example}

Imposed values of the degrees of freedom translate into the linear relations among the B-coefficients presented in the following theorem; see e.g. \cite[Thm. 2.15]{Lai.Schumaker07}.
\begin{theorem}\label{thm:dirder}
Let $1\leq m\leq d$, and suppose we are given, for $i = 1,\ldots,m$, a vector $\uu_i$ with directional coordinates $a^{(i)} := \big(a^{(i)}_1 , a^{(i)}_2 , a^{(i)}_3\big)$. Then
\[ \nabla_{\uu_m} \cdots \nabla_{\uu_1} p (\vv) =
\frac{d!}{(d-m)!} \sum_{i+j+k=d-m} c^{(m)}_{ijk} B_{ijk}^{d-m}(\vv),
\]
where $c^{(0)}_{ijk} = c_{ijk}$ and
\[
c^{(m)}_{ijk} =
a^{(m)}_1 c^{(m-1)}_{i+1,j,k} + a^{(m)}_2 c^{(m-1)}_{i,j+1,k} + a^{(m)}_3 c^{(m-1)}_{i,j,k+1},
\qquad m = 1,\ldots, d.
\]
\end{theorem}

As a consequence, for any $r\leq d$, specifying all derivatives up to order $r$ at some vertex $\vv$ is equivalent to specifying the B-coefficients in the disk $D_r(\vv)$ \cite[Thm. 2.18]{Lai.Schumaker07}.

\begin{figure}
\subfloat[]{\includegraphics[scale=0.7]{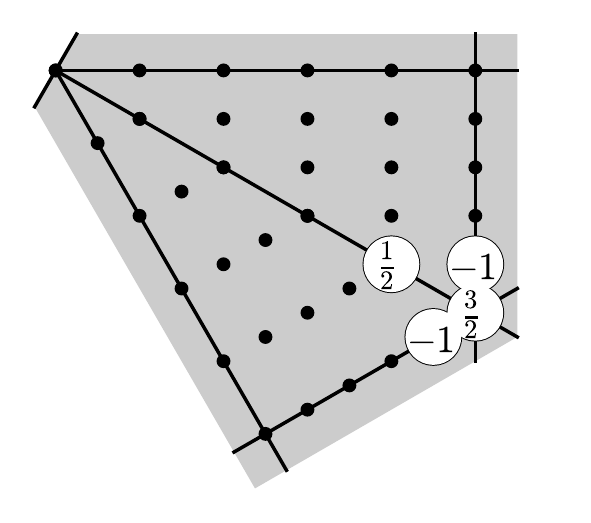}\label{fig:bivariateC1}}
\subfloat[]{\includegraphics[scale=0.7]{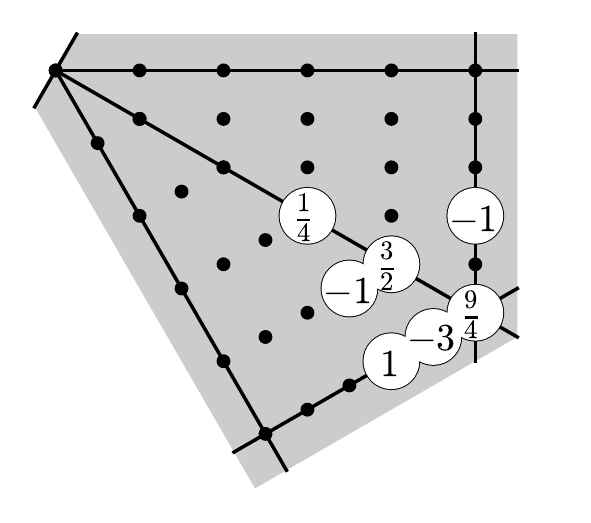}\label{fig:bivariateC2}}
\subfloat[]{\includegraphics[scale=0.7]{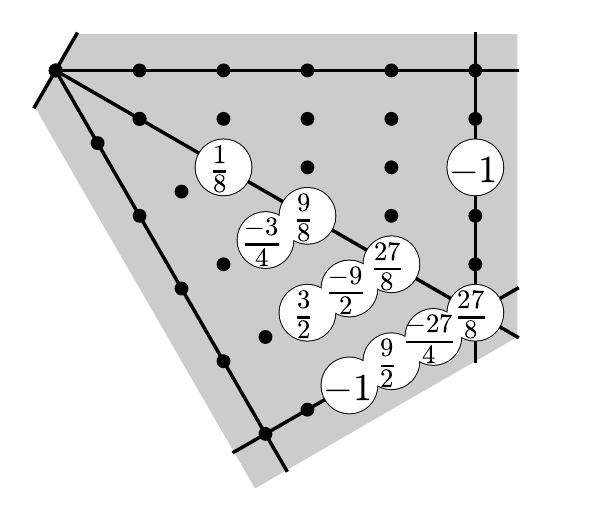}\label{fig:bivariateC3}}\\
\subfloat[]{\raisebox{-0.55em}{\includegraphics[scale=0.7]{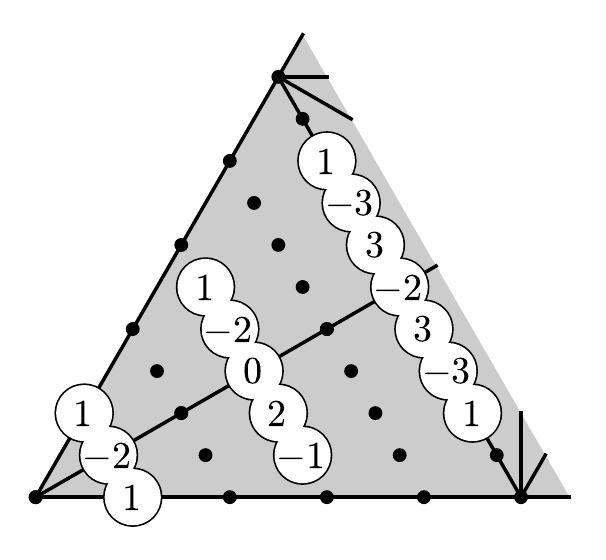}}\label{fig:univariate1}}
\subfloat[]{\includegraphics[scale=0.7]{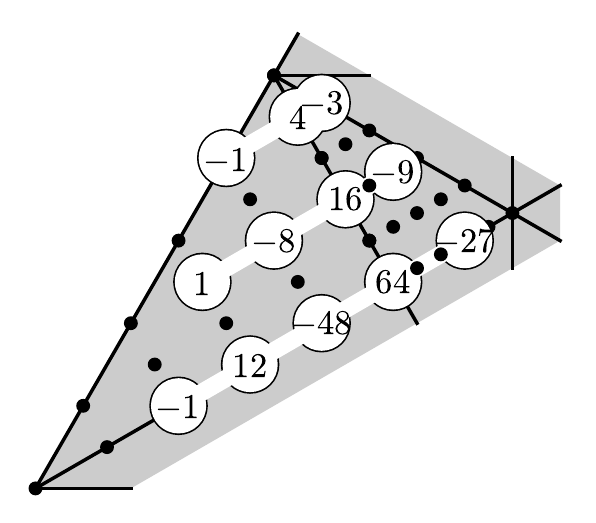}\label{fig:univariate2}}
\subfloat[]{\includegraphics[scale=0.7]{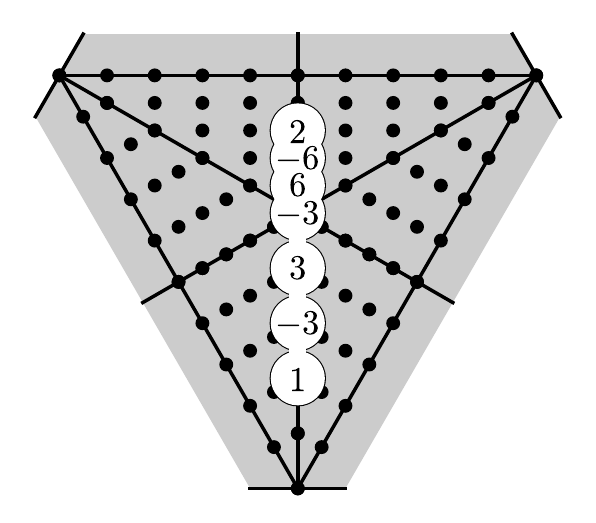}\label{fig:univariate3}}

\caption{Setting to zero linear combinations of the B-coefficients with weights in the figures gives various smoothness conditions \eqref{eq:smoothnessconditions} across the edges $\langle \vv_6, \vv_{10}\rangle$ (a, b, c), $\langle\vv_1,\vv_7 \rangle$ (d), $\langle\vv_6, \vv_7 \rangle$ (e), and along the line segment $\langle \vv_9, \vv_4 \rangle$ (f) in Figure~\ref{fig:PS12-Labels12}. The stencils in (a, b, c, d, e) can be moved along the shared edge.
}\label{fig:bivariateC1C2C3}
\end{figure}

\section{Macro-elements on Powell-Sabin splits}\label{sec:macroelement}
\noindent For the purpose of constructing approximations by piecewise quadratics, Powell and Sabin introduced two methods to split each triangle in a triangulation $\PS$ \cite{Powell.Sabin77}. For the first of these, one considers a new vertex $\vv_T$ in the interior of each triangle $T=\langle \vv_1, \vv_2, \vv_3 \rangle$ in such a way that, for any triangles $T, \tilde{T}$ sharing an edge $e$, the line segment $\langle \vv_T, \vv_{\tilde{T}} \rangle$ intersects $e$ in a new point $\ww_e$. For any edge $e$ on the boundary, one writes $\ww_e$ for the midpoint of $e$. Inserting, for any edge $e$ and vertex $\vv_i$ of a triangle $T$, the new vertices $\vv_T$, $\ww_e$ and edges $\langle \vv_T, \ww_e \rangle, \langle \vv_T, \vv_i \rangle$ into $\Delta$, one arrives at a refined triangulation $\PSA$ known as the \emph{Powell-Sabin 6-split} of $T$; see Figure \ref{fig:PS12-Labels6}.

How should one choose the points $\vv_T$ so that each $\ww_e$ is well defined? If all triangles are acute one can choose $\vv_T$ to be the circumcenter of $T$, and in general one can choose $\vv_T$ to be the incenter of $T$ \cite{Lai.Schumaker03}. Alternatively, one can take the barycenter $\vv_T = (\vv_1 + \vv_2 + \vv_3)/3$ of each triangle $T$, but this is only possible if the line-segment joining the barycenters of two adjacent triangles intersect the common edge. Alfeld and Schumaker \cite{Alfeld.Schumaker02} avoid this geometric constraint by, for any edge $e$ shared by two adjacent triangles $T, \tilde{T}$, choosing $\ww_e$ freely, no longer requiring $\vv_T, \ww_e, \vv_{\tilde{T}}$ to be collinear. This has the additional advantage that the choice of $\vv_T$ no longer affects the geometry of the split in any neighbouring macrotriangle.

The latter advantage is shared by the second split introduced by Powell and Sabin, which is constructed as follows. Given a triangle $T = \langle \vv_1, \vv_2, \vv_3\rangle$, write $e_1 := \langle \vv_2, \vv_3\rangle$, $e_2 := \langle \vv_3, \vv_1\rangle$, and $e_3 := \langle \vv_1, \vv_2\rangle$ for its edges. The \emph{midpoint} of any edge $e = \langle \vv, \vv'\rangle$ is the average $\ww_e := (\vv + \vv')/2$ of its endpoints, and its \emph{quarterpoints} are the averages $(3\vv + \vv')/4$ and $(\vv + 3\vv')/4$ of the midpoint with the endpoints. We construct a 6-split by inserting a vertex $\vv_T$ at the barycenter, which we connect to the old vertices and the midpoints. Connecting the midpoints we arrive at the \emph{Powell-Sabin 12-split} of $T$; see Figure \ref{fig:PS12-Labels12}.

Let $\PS$ be a triangulation of some domain $\Omega\subset \RR^2$. Replacing every triangle in $\PS$ by a 6-split (resp. 12-split) results in a finer triangulation $\PSA$ (resp. $\PSB$). We write $\SSS_6$ (resp. $\SSS_{12}$) for the space of piecewise quintic polynomials with global $C^2$-smoothness on $\PSA$ (resp. $\PSB$) and $C^3$-smoothness on each macrotriangle in $\PS$. We consider a set $\Lambda_6$ of functionals on $\SSS_6$, comprising point evaluations and partial derivatives up to order three at any vertex $v$ of $\PS$. We also consider the superset $\Lambda_{12}\supset \Lambda_6$ of functionals on $\SSS_{12}\supset \SSS_6$, which in addition contains, for each edge $e$ in $\PS$, a first-order cross-boundary derivative at the midpoint and a second-order cross-boundary derivative at each quarter point; see Figure \ref{fig:PS12-1}.

On a single triangle, one has $\dim(\SSS_{12}) = 39$ and $\dim(\SSS_6) = 30$, matching the cardinalities of $\Lambda_{12}$ and $\Lambda_6$. This can be computed numerically using a JAVA applet written by Peter Alfeld described in \cite{Alfeld00} and available at \verb§http://www.math.utah.edu/~pa/MDS/§. The dimension of the 6-split follows immediately from an ``interior cell formula'' \cite[Thm. 9.3]{Lai.Schumaker07}, while the dimension of the 12-split follows from a formula for the dimension for general degree $d$ and smoothness $r$, which we plan on reporting elsewhere.

\begin{remark}\label{rem:C3}
In Section \ref{sec:subdivision} we derive a Hermite subdivision scheme for computing any spline $s\in \SSS_{12}$ from the initial data
\begin{equation}\label{eq:InitialData12}
\lambda(s) = s_\lambda,\qquad \lambda\in \Lambda_{12}.
\end{equation}
Along any edge $e$ in $\Delta$, a third-order cross-boundary derivative at the midpoint $\ww_e$ cannot be determined from just the functionals along $e$, as this would imply global $C^3$-smoothness of $s$. Since any spline $s\in \SSS_{12}$ has global $C^2$-smoothness and $C^3$-smoothness within each macrotriangle, the subdivision rules for the other derivatives at $\ww_e$ can only involve initial data along~$e$.
\end{remark}

The following theorem shows that any spline $s\in \SSS_{12}$ is uniquely determined from the initial data \eqref{eq:InitialData12}.

\begin{theorem}\label{thm:Macro-Element12}
The set $\Lambda_{12}$ forms a nodal minimal determining set for $\SSS_{12}$.
\end{theorem}

\begin{figure}
\begin{center}
\noindent\includegraphics[scale=0.71]{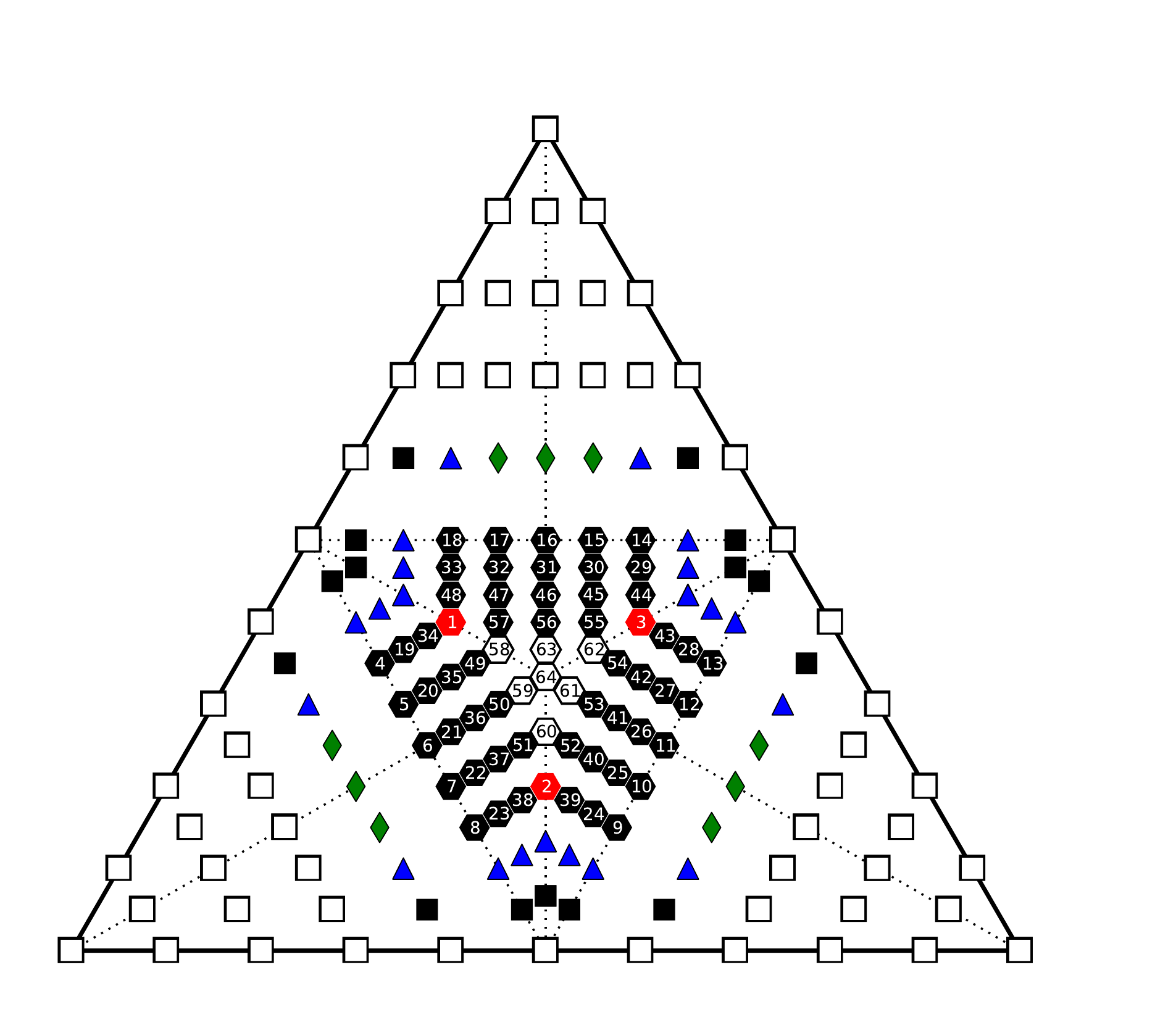}
\end{center}
\caption{A schematic depiction of successively determined B-coefficients of type $\square, \blacksquare,{\color{blue}\blacktriangle}, {\color{Green}\blacklozenge}, \hexagon$ in the proof of Theorem~\ref{thm:Macro-Element12}. }\label{fig:Proof-MDS12}
\end{figure}

\begin{proof}
For any vertex $\vv$ of $\PS$, the derivatives up to order three at $\vv$ determine the disk $D_3(\vv)$. Along each edge $e = \langle \vv_1, \vv_2 \rangle$ of $\PS$, we therefore know the initial and final four B-coefficients. Therefore, since the space of $C^3$-smooth quintic univariate splines along $e$ has dimension 8, the restriction of $s$ to $e$ is determined. Thus the nodal data at the corners determine the B-coefficients marked $\square$ in Figure~\ref{fig:Proof-MDS12}.

Let $\mm$ and $\tt$ be the medial and tangential directions associated to $e$. Since the space of $C^2$-smooth quartic univariate splines along $e$ has dimension 7, the restriction of the cross-boundary derivative $\nabla_{\mm} s$ to $e$ is determined by its value at the mid-point $\ww_e$ and the 3 degrees of freedom $\nabla_{\mm} s, \nabla_{\tt\mm} s, \nabla_{\tt\tt\mm} s$ at the corners $\vv_1,\vv_2,\vv_3$. Specifying, in addition, a first-order cross-boundary derivative at each midpoint therefore determines the B-coefficients marked $\blacksquare$ in Figure \ref{fig:Proof-MDS12}.

Similarly, since the space of $C^1$-smooth cubic univariate splines along $e$ has dimension 6, the restriction of $\nabla^2_{\mm} s$ to $e$ is determined by its value at the 2 quarterpoints of $e$ and the 2 degrees of freedom $\nabla^2_{\mm} s, \nabla^2_{\mm} \nabla_{\tt} s$ at the corners $\vv_1,\vv_2,\vv_3$. Specifying second-order cross-boundary derivatives at the quarterpoints therefore determines the B-coefficients marked ${\color{blue}\blacktriangle}$ in Figure~\ref{fig:Proof-MDS12}.

In order to show that these 39 conditions are independent, we need to determine the remaining B-coefficients from the smoothness conditions. For instance, using the smoothness conditions with coefficients shown in Figure~\ref{fig:univariate1}, one determines the B-coefficients marked ${\color{Green}\blacklozenge}$ in Figure \ref{fig:Proof-MDS12}.

Let us refer to the coefficients thus far as \emph{known}. The remaining \emph{unknown} coefficients are numbered $c_1,\ldots,c_{64}$ as in Figure~\ref{fig:Proof-MDS12}. By Theorem~\ref{thm:dirder}, determining $c_1$ is equivalent to determining a third-order cross-boundary derivative at the midpoint $\vv_6$. By Remark \ref{rem:C3}, this derivative cannot be determined from only the initial data along the edge $\langle \vv_1, \vv_2\rangle$. For the time being, therefore, let us express $c_1,\ldots, c_{64}$ in terms of $c_1, c_2, c_3$ and the known coefficients, writing $a\equiv b$ if $a - b$ is a linear combination of known coefficients.

Using the univariate $C^3$-condition from Figure \ref{fig:univariate2} across the edge $\langle \vv_6, \vv_7\rangle$, one expresses $c_4$ in terms of $c_1$. Similarly one expresses $c_8$ in terms of $c_2$. The univariate smoothness conditions from Figure~\ref{fig:univariate1} along $\langle\vv_4, \vv_6\rangle$ determine the remaining B-coefficients $c_5, c_6, c_7$ along this edge. Using the univariate smoothness conditions from Figure~\ref{fig:univariate2} across the edge $\langle \vv_6, \vv_7\rangle$, and similar relations for $\langle \vv_4, \vv_7\rangle$, one obtains
\begin{align*}
\left(\frac34\right)^3 c_1 \equiv c_4 & \equiv \frac34 c_{19}\equiv \left(\frac34\right)^2 c_{34} \\ 
\left(\frac34\right)^3 (c_1 + \frac12 c_2) \equiv c_5 & \equiv \frac34 c_{20}\equiv \left(\frac34\right)^2 c_{35}\equiv \left(\frac34\right)^3 c_{49}\\
\left(\frac34\right)^4 (c_1 + c_2) \equiv c_6 & \equiv \frac34 c_{21}\equiv \left(\frac34\right)^2 c_{36}\equiv \left(\frac34\right)^3 c_{50}\\ 
\left(\frac34\right)^3 (\frac12 c_1 + c_2) \equiv c_7 & \equiv \frac34 c_{22}\equiv \left(\frac34\right)^2 c_{37}\equiv \left(\frac34\right)^3 c_{51}\\
\left(\frac34\right)^3 c_2 \equiv c_8 & \equiv \frac34 c_{23}\equiv \left(\frac34\right)^2 c_{38}
\end{align*}
Similar expressions hold for the other unknown B-coefficients outside of $D_1(\vv_{10})$. Using the $C^1$-condition from Figure~\ref{fig:bivariateC1} across $\langle \vv_6, \vv_{10}\rangle$ (and similarly for $\langle \vv_4, \vv_{10}\rangle$ and $\langle \vv_5, \vv_{10}\rangle$),
\[ c_{58} \equiv c_1 + \frac13 c_2 + \frac13 c_3,\quad c_{60}\equiv c_2 + \frac13 c_1 + \frac13 c_3,\quad c_{62}\equiv c_3 + \frac13 c_1 + \frac13 c_2, \]
and taking averages of these coefficients (as the $C^1$-condition in Figure~\ref{fig:univariate1}),
\[ c_{59} \equiv \frac{2c_1 + 2c_2 + c_3}{3},\quad c_{61} \equiv \frac{c_1 + 2c_2 + 2c_3}{3},\quad c_{63} \equiv \frac{2c_1 + c_2 + 2c_3}{3}. \]
Using any $C^1$-condition at $\vv_{10}$ (e.g. Figure~\ref{fig:bivariateC1}), we find $c_{64} = \frac59 (c_1 + c_2 + c_3)$.

At this point all unknown B-coefficients are expressed in terms of $c_1,c_2,c_3$. The univariate $C^3$-condition along $\langle \vv_4, \vv_9 \rangle$ from Figure~\ref{fig:univariate3} (and similarly for $\langle \vv_5, \vv_7 \rangle$, $\langle \vv_6, \vv_8 \rangle$) yield
\[ c_1 \equiv (c_2 + c_3)/2,\qquad c_2 \equiv (c_1 + c_3)/2,\qquad c_3 \equiv (c_1 + c_2)/2,\]
so that $c_1\equiv c_2\equiv c_3$. Substituting $c_1\equiv c_2\equiv c_3$ in the $C^3$-condition from Figure~\ref{fig:bivariateC3},
\begin{align*}
0 &\equiv - c_{46} - \frac34 c_{34} + \frac98 c_1 + \frac32 c_{35} - \frac92 c_{49} + \frac{27}{8} c_{58} - c_{36} + \frac92 c_{50}\\
 &\hskip 1cm - \frac{27}{4}c_{59} + \frac{27}{8}c_{64} \equiv \frac{18}{16} c_1
\end{align*}
determines the remaining degree of freedom $c_1$ in terms of the known B-coefficients.

Finally, since the univariate splines $s|_e, \nabla_{\mm} s|_e$, and $\nabla^2_{\mm} s|_e$ are determined from just the functionals along~$e$, the values, first, and second derivatives of the splines on the adjacent triangles agree along $e$. It follows that $\Lambda_{12}$ defines a spline $s$ that is $C^2$-smooth on $\PS$ and $C^3$-smooth on each macro-triangle in $\PS$.
\end{proof}

Similarly, any spline $s\in \SSS_6$ is uniquely determined from the initial data
\begin{equation}\label{eq:InitialData6}
\lambda(s) = s_\lambda,\qquad \lambda\in \Lambda_6,
\end{equation}
which follows from combining Theorem 2.18 and 7.9 in \cite{Lai.Schumaker07}:
\begin{theorem}\label{thm:Macro-Element6}
The set $\Lambda_6$ forms a nodal minimal determining set for $\SSS_6$.
\end{theorem}

\section{Derivation of a subdivision scheme}\label{sec:subdivision}
\noindent It has been observed by many (see e.g. \cite{DynLyche98, Oswald92}) that the geometry of the 12-split lends itself for subdivision into four subtriangles, namely $(\vv_1, \vv_4, \vv_6)$, $(\vv_2, \vv_5, \vv_4)$, $(\vv_3, \vv_6, \vv_5)$, and $(\vv_4, \vv_5, \vv_6)$ in Figure \ref{fig:PS12-Labels12}. Each of these is (refinable to) a 6-split. Moreover, splitting the 6-split in Figure \ref{fig:PS12-Labels6} into four subtriangles yields four new 6-splits, etc. Using this observation, one can derive a Hermite subdivision scheme to compute the $C^1$-smooth quadratic splines from values and first derivatives at the corners and cross-boundary derivatives at the midpoints \cite{DynLyche98}.

In this section we derive a similar scheme for computing any spline $s \in \SSS_{12}$ from the initial data~\eqref{eq:InitialData12}. After computing the values and derivatives at the midpoints in the ``initialization step'' in Figure~\ref{fig:PS12-2}, one obtains a refined description of $s$ by repeatedly applying the ``subdivision step'' in Figure~\ref{fig:PS12-3}.

\begin{figure}
\begin{center}
\subfloat[]{\includegraphics[scale=0.69]{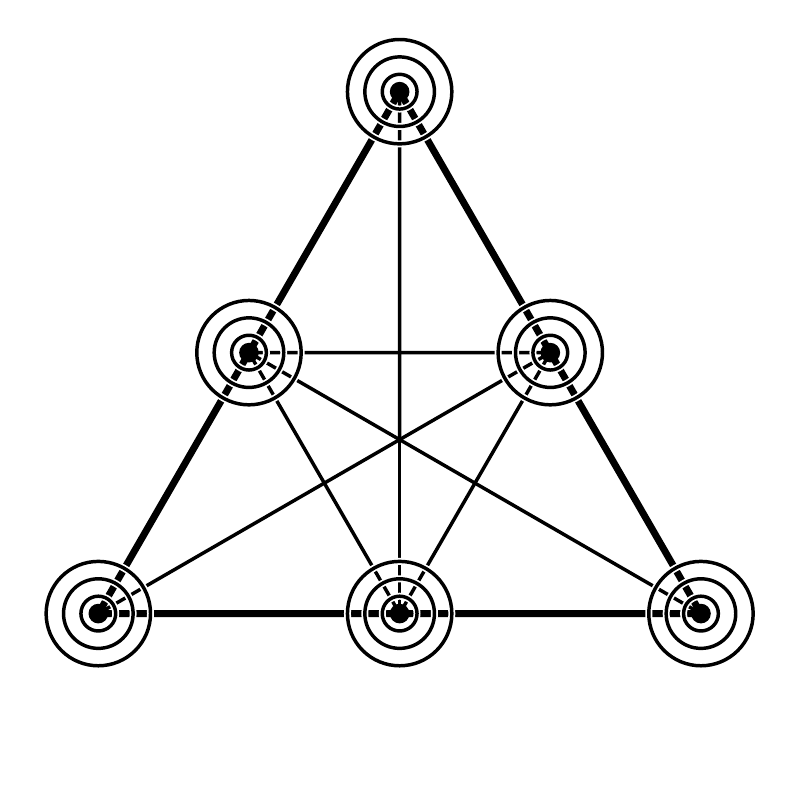}\label{fig:PS12-2}}\qquad 
\subfloat[]{\includegraphics[scale=0.69]{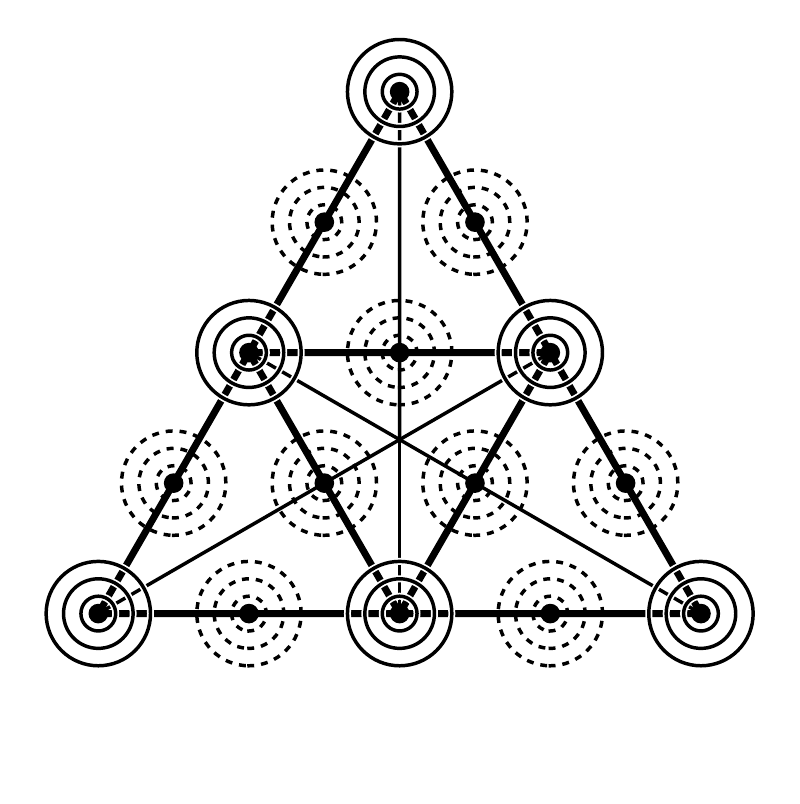}\label{fig:PS12-3}}
\end{center}
\caption{Known data after the `initialization step' (left) and additional known data (dashed) after the `subdivision step' on the four subtriangles (right).}\label{fig:PS12}
\end{figure}

\begin{remark}
For two adjacent triangles with common edge $\langle \vv_Y, \vv_Z \rangle$, specifying cross-boundary derivatives in the normal direction $\nn_X$ at the midpoints and quarterpoints has the advantage of only depending on the common edge. However, since tangential derivatives at the midpoints and quarterpoints can easily be computed from the tangential derivatives at the corners $\vv_Y, \vv_Z$ (see \eqref{eq:m-vs.-n} and the rules below), any other directional derivative can be computed by a simple change of basis.

To keep the rules as short as possible, we specify directional derivatives in the medial directions $\mm_X$ and tangential directions~$\tt_X$. Because the edges in the 12-split are divided at the midpoints, each of the four subdivided triangles in Figure \ref{fig:PS12-3} will have the same tangential and medial vectors, albeit scaled by a factor $\pm\frac12$. Thus directional derivatives computed on previous levels can be reused.
\end{remark}

It follows from Theorem \ref{thm:Macro-Element12} (resp. Theorem \ref{thm:Macro-Element6}) that the B-coefficients of any spline in $S_{12}$ (resp. $S_6$) can be uniquely determined from the initial data \eqref{eq:InitialData12} (resp. \eqref{eq:InitialData6}). Moreover, Theorem \ref{thm:dirder} implies that any directional derivative of any order of the spline at any point can be expressed in terms of the B-coefficients. It follows that there exist initialization rules and subdivision rules for evaluating the spline as described above. In the next two sections we derive these rules.

By Remark~\ref{rem:C3}, the initialization and subdivision rule for the third-order cross-boundary derivative at the midpoint $\ww_e$ is expected to be complicated, involving functionals with carrier outside of $e$. Because of this, the formulas for the case of (locally) $C^3$-quintics are considerably more complex than those for the $C^1$-quadratics. As the computations are too large to carry out by hand, we made use of the computer algebra system \Sage~\cite{Sage}. The resulting worksheet with implementation and examples can be downloaded from the website of the second author~\cite{WebsiteGeorg}.

\subsection*{Initialization}
\noindent On a single triangle $\PS$, with the 252 unknown B-coefficients
\[ \{c_{ijk}^l\ :\ l = 1, 2, \ldots, 12,\ i + j + k = 5\}, \]
the initial conditions \eqref{eq:InitialData12} and smoothness conditions \eqref{eq:smoothnessconditions} for the 15 interior edges in $\PSB$ form a linear system with $39 + 15\cdot (6 + 5 + 4 + 3) = 309$ equations. By the linear independence of the functionals in $\Lambda_{12}$, this system has a unique solution. Solving the system and applying Theorem \ref{thm:dirder}, we find initialization rules for computing the data in Figure \ref{fig:PS12-2} from the initial data~\eqref{eq:InitialData12}.

Given are, at the corners $\vv_A, \vv_B, \vv_C$, values and derivatives of $f$ up to order three, first-order cross-boundary derivatives $f^{\mm_C}_{AB}, f^{\mm_A}_{BC}, f^{\mm_B}_{CA}$ at the midpoints, and second-order cross-boundary derivatives
\[ f^{\mm_C \mm_C}_{AAB}, f^{\mm_C \mm_C}_{ABB},\qquad f^{\mm_A \mm_A}_{BBC}, f^{\mm_A \mm_A}_{BCC},\qquad f^{\mm_B \mm_B}_{CCA}, f^{\mm_B \mm_B}_{CAA}\]
at the quarter points as in Figure \ref{fig:PS12-1}. With the short-hands $\tt := \tt_C, \mm := \mm_C$, and \eqref{eq:shorthand}, the value and derivatives at the midpoint $\vv_{AB}$ are determined by the rules (see the worksheet)
\begin{align*}
f_{AB} = & \frac{1}{2} f_{A+B} + \frac{7}{40}f^{\tt}_{A-B} + \frac{1}{40} f^{\tt\tt}_{A+B} + \frac{1}{640} f^{\tt\tt\tt}_{A-B}\\ 
f^{\tt}_{AB} = & -\frac{5}{2} f_{A-B} - \frac{3}{4} f^\tt_{A+B} - \frac{3}{32} f^{\tt\tt}_{A-B} - \frac{1}{192}f^{\tt\tt\tt}_{A+B}\\ 
f^{\tt\tt}_{AB}  = & -2f^\tt_{A-B} - \frac12 f^{\tt\tt}_{A+B} - \frac{1}{24} f^{\tt\tt\tt}_{A-B}\\ 
f^{\tt\mm}_{AB}  = & -2f^\mm_{A-B} - \frac12 f^{\tt\mm}_{A+B} - \frac{1}{24} f^{\tt\tt\mm}_{A-B}\\ 
f^{\mm\mm}_{AB}  = & f^{\mm\mm}_{AAB + ABB} - \frac12 f^{\mm\mm}_{A+B} - \frac{1}{16} f^{\tt\mm\mm}_{A-B}\\ 
f^{\tt\tt\tt}_{AB} = & 120f_{A-B} + 60f^\tt_{A+B} + \frac{21}{2} f^{\tt\tt}_{A-B} + \frac{3}{4} f^{\tt\tt\tt}_{A+B}\\
f^{\tt\tt\mm}_{AB} = & -48 f^\mm_{AB} + 24 f^\mm_{A+B} + 6 f^{\tt\mm}_{A-B} + \frac{1}{2} f^{\tt\tt\mm}_{A+B}\\ 
f^{\tt\mm\mm}_{AB} = & -8f^{\mm\mm}_{AAB-ABB} + 4f^{\mm\mm}_{A-B} + \frac12 f^{\tt\mm\mm}_{A+B}\\
f^{\mm\mm\mm}_{AB} = &\  45 f_{A+B} + 45 f^\mm_{A+B} + 36 f^\tt_{A-B} - \frac{217}{16} f^{\mm \mm}_{A+B} + \frac{153}{16} f^{\mm\tt}_{A-B}\\
 & + \frac{567}{64} f^{\tt\tt}_{A+B} + \frac{25}{64} f^{\mm\mm\mm}_{A+B} - \frac{251}{128} f^{\mm\mm\tt}_{A-B} + \frac{43}{256} f^{\mm\tt\tt}_{A+B} + \frac{303}{512} f^{\tt\tt\tt}_{A-B}\\
 & - 90 f_C - 24 f^\mm_C - \frac{23}{8} f^{\mm \mm}_C - \frac{135}{32} f^{\tt\tt}_C - \frac{5}{32} f^{\mm\mm\mm}_C - \frac{79}{128} f^{\mm\tt\tt}_C\\
 & -108 f^\mm_{AB} + 48 f^{\mm_A}_{BC} + 48 f^{\mm_B}_{CA}\\
 &   + 15 f^{\mm\mm}_{AAB+ABB} -7 f^{\mm_A\mm_A}_{BBC} + f^{\mm_A\mm_A}_{BCC} -7 f^{\mm_B\mm_B}_{CAA} +  f^{\mm_B\mm_B}_{CCA}.
\end{align*}
Note that only the third-order cross-boundary derivative $f^{\mm\mm\mm}_{AB}$ involves functionals with carrier outside of $\langle \vv_A, \vv_B\rangle$, as explained in Remark \ref{rem:C3}.

\subsection*{Subdivision}
\noindent On a single triangle $\PS$, with the 126 unknown B-coefficients
\[ \{c_{ijk}^l\ :\ l = 1, 2, \ldots, 6,\ i + j + k = 5\},\]
the initial conditions \eqref{eq:InitialData6} and smoothness conditions \eqref{eq:smoothnessconditions} for the 6 interior edges in $\PSA$ form a linear system with $30 + 6\cdot (6 + 5 + 4 + 3) = 138$ equations. By the linear independence of the functionals in $\Lambda_6$, this system has a unique solution. Solving the system and applying Theorem \ref{thm:dirder}, we find subdivision rules for computing the dashed data in Figure \ref{fig:PS12-3} from the data in Figure~\ref{fig:PS12-2} (see the worksheet),
\begin{align*}
f_{AB}             = &\,  \frac12 f_{A+B} + \frac{7}{40} f^{\tt}_{A-B} + \frac{1}{40} f^{\tt\tt}_{A+B} + \frac{1}{640} f^{\tt\tt\tt}_{A-B}\\
f^{\tt}_{AB}       = &\, -\frac52 f_{A-B} - \frac34      f^{\tt}_{A+B} - \frac{3}{32} f^{\tt\tt}_{A-B} - \frac{1}{192} f^{\tt\tt\tt}_{A+B}\\
f^{\mm}_{AB}       = &\, \frac12 f^\mm_{A+B} + \frac{5}{32} f^{\mm\tt}_{A-B} + \frac{1}{64} f^{\mm\tt\tt}_{A+B} \\
f^{\tt\tt   }_{AB} = &\, -2 f^\tt_{A-B} - \frac{1}{2} f^{\tt\tt}_{A+B} - \frac{1}{24} f^{\tt\tt\tt}_{A-B} \\
f^{\mm\tt   }_{AB} = &\, -2 f^\mm_{A-B} - \frac{1}{2} f^{\mm\tt}_{A+B} - \frac{1}{24} f^{\mm\tt\tt}_{A-B} \\
f^{\mm\mm   }_{AB} = &\, \frac12 f^{\mm\mm}_{A+B} + \frac18 f^{\tt\mm\mm}_{A-B}\\
f^{\tt\tt\tt}_{AB} = &\,  120 f_{A-B} +  60 f^\tt_{A+B} + \frac{21}{2} f^{\tt\tt}_{A-B} + \frac34 f^{\tt\tt\tt}_{A+B} \\
f^{\mm\tt\tt}_{AB} = &\, -\frac14 f^{\mm\tt\tt}_{A+B} - \frac32 f^{\mm\tt}_{A-B}\\
f^{\mm\mm\tt}_{AB} = &\, -\frac14 f^{\mm\mm\tt}_{A+B} - \frac32 f^{\mm\mm}_{A-B}\\
f^{\mm\mm\mm}_{AB} = & + 45 f_{A+B}  + 18 f^{\tt}_{A-B} - 21 f^{\mm}_{A+B} + \frac{45}{16} f^{\tt\tt}_{A+B} - \frac{63}{8} f^{\tt\mm}_{A-B} \\
  & + \frac{15}{4} f^{\mm\mm}_{A+B} + \frac{3}{16} f^{\tt\tt\tt}_{A-B} - \frac78 f^{\tt\tt\mm}_{A+B} + \frac54 f^{\tt\mm\mm}_{A-B} + \frac14 f^{\mm\mm\mm}_{A+B} \\
  &  - 90 f_C - 48 f^{\mm}_C + \frac98 f^{\tt\tt}_C - \frac{21}{2} f^{\mm\mm}_C  + \frac14 f^{\tt\tt\mm}_C - f^{\mm\mm\mm}_C.
\end{align*}

\begin{remark}
Taking convex combinations of the rules for the $C^1$-quadratic and $C^3$-quintic schemes, one arrives at a family of Hermite interpolatory subdivision schemes. A general member of this family will produce limit functions that are not piecewise polynomial.
\end{remark}

\section{Implementation and experimentation}
\noindent In this section we apply the scheme to compute some example splines in some example triangulations. For details we refer to the \Sage~worksheet.

\begin{figure}
\begin{center}
\includegraphics[scale=0.10, bb = 80 72 670 850, clip = True]{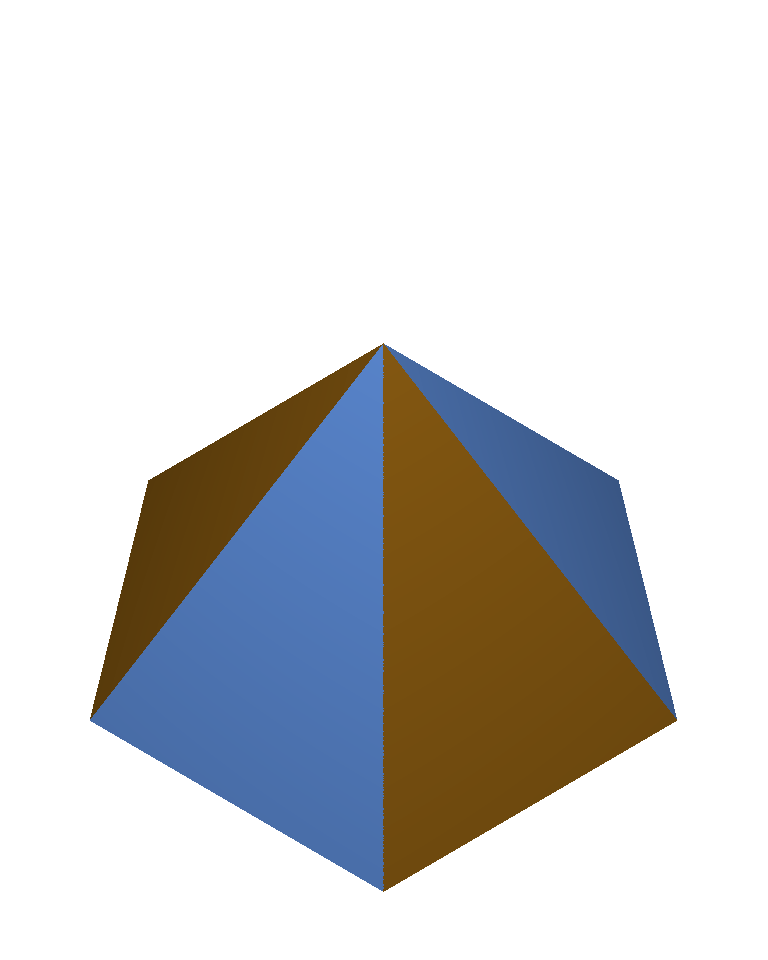}\quad
\includegraphics[scale=0.10, bb = 80 72 670 850, clip = True]{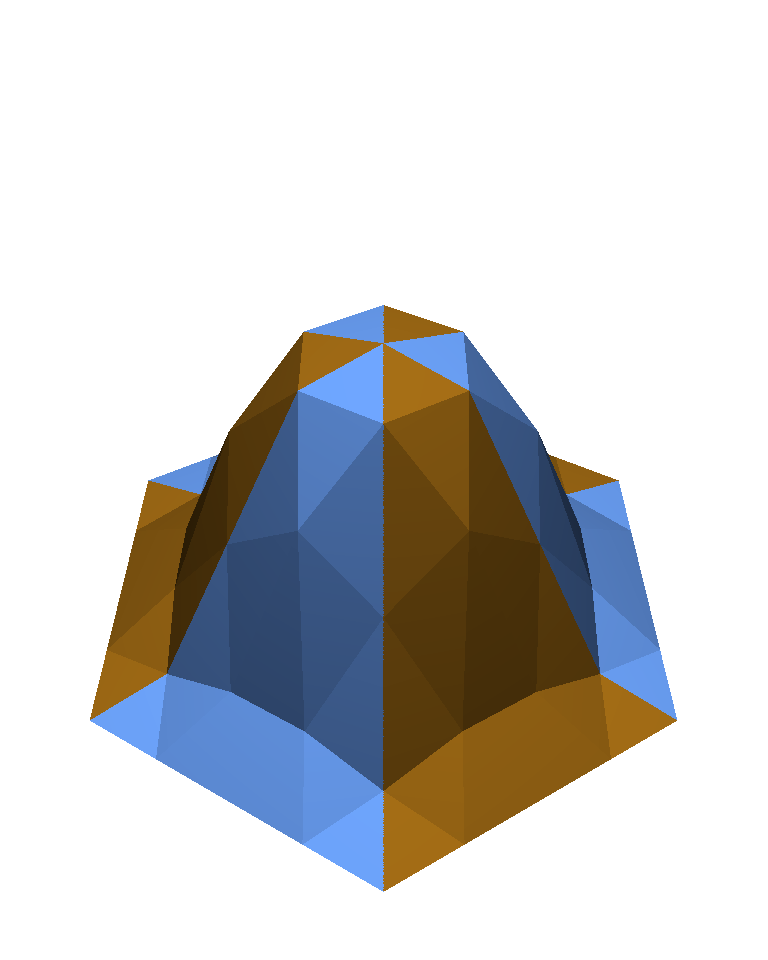}\quad
\includegraphics[scale=0.10, bb = 80 72 670 850, clip = True]{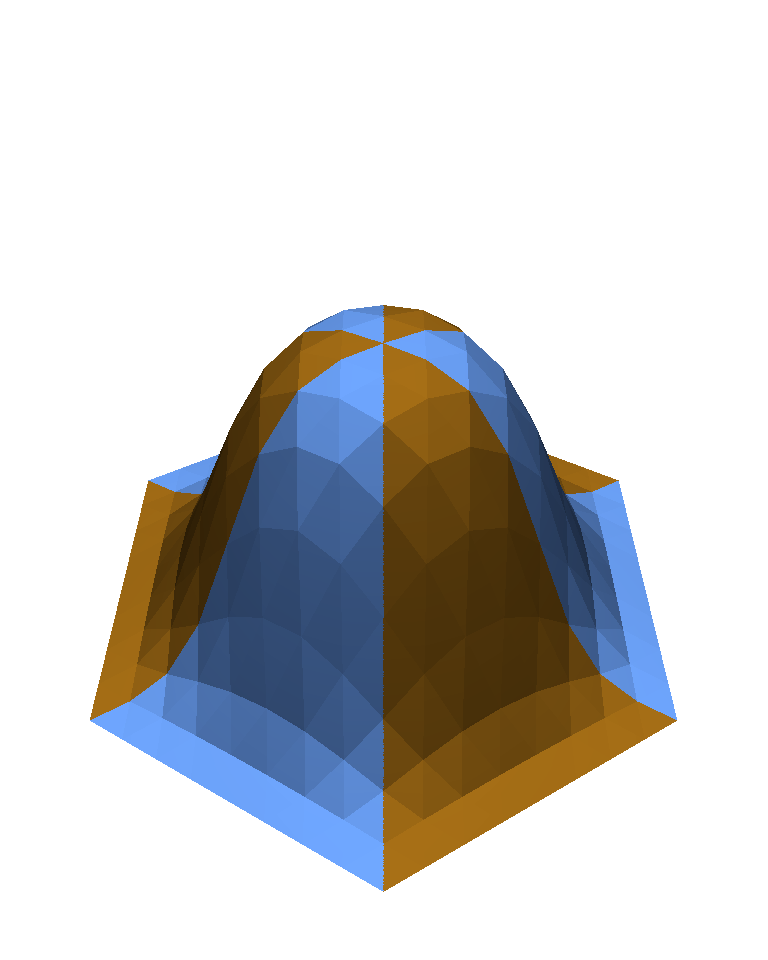}\quad
\includegraphics[scale=0.10, bb = 80 72 670 850, clip = True]{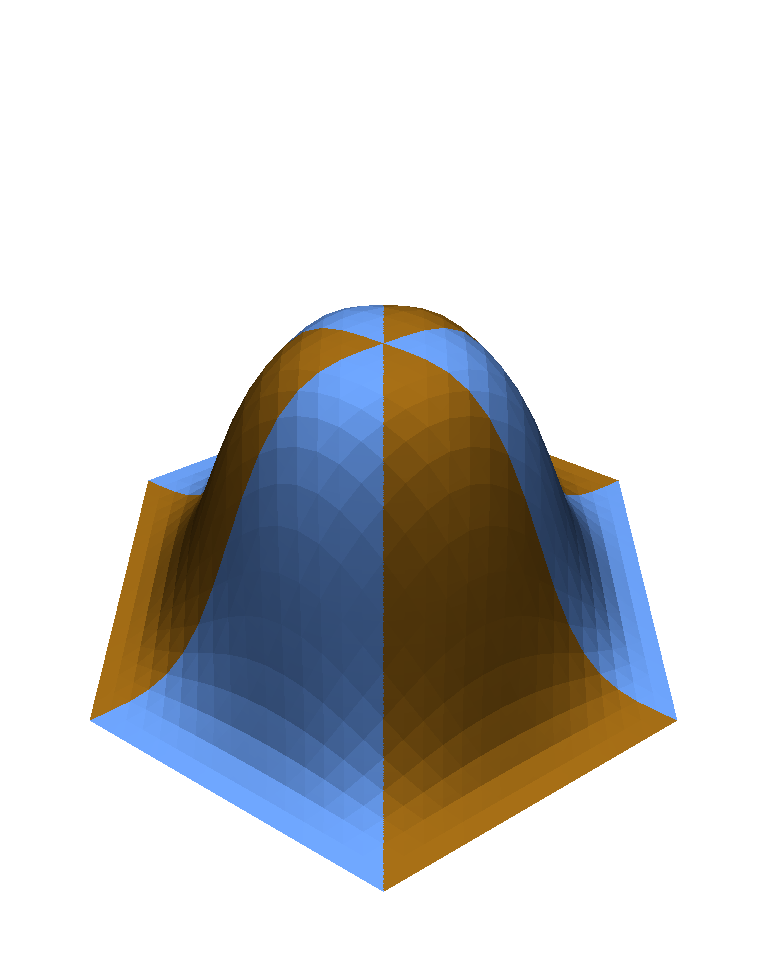}\quad
\includegraphics[scale=0.10, bb = 80 72 670 850, clip = True]{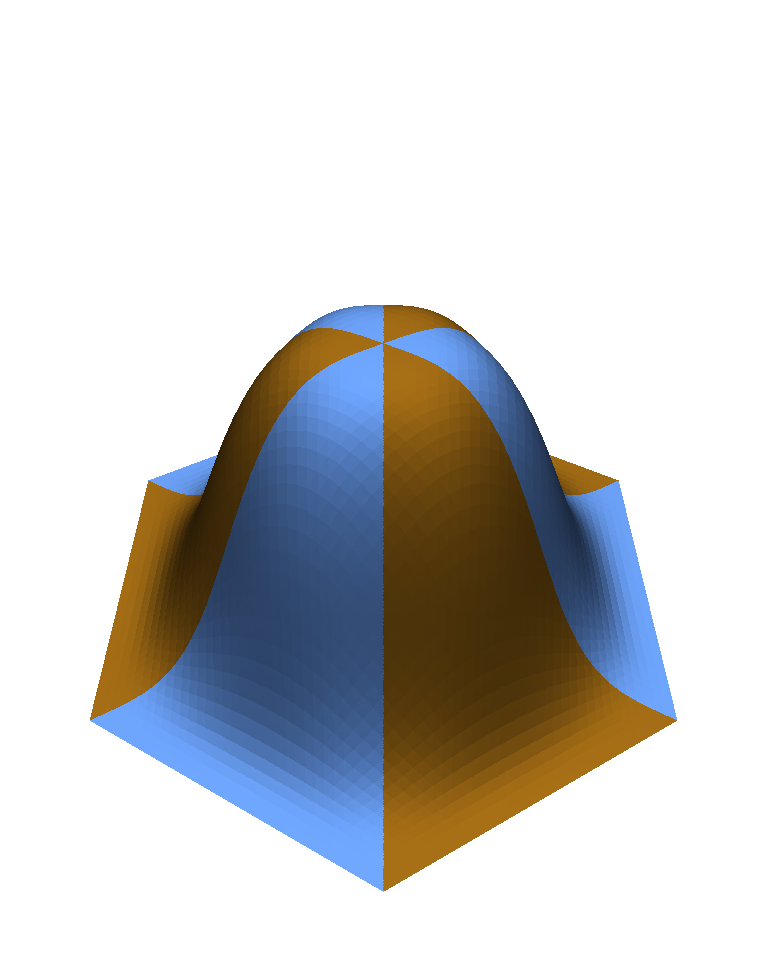}\\
\includegraphics[scale=0.10, bb = 80 72 670 850, clip = True]{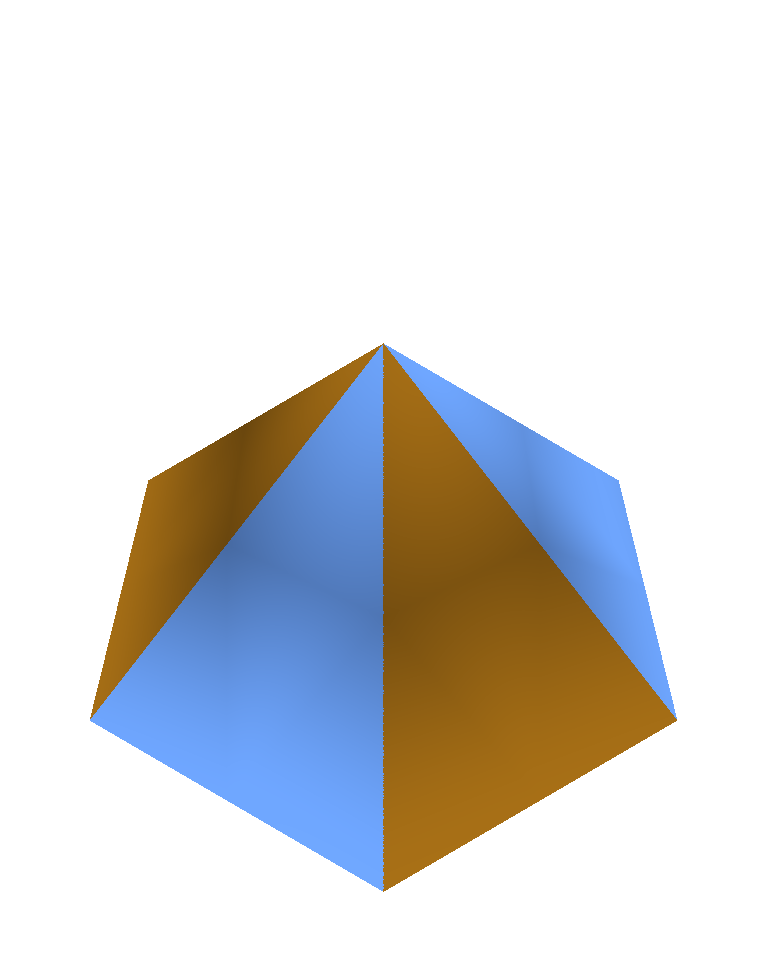}\quad
\includegraphics[scale=0.10, bb = 80 72 670 850, clip = True]{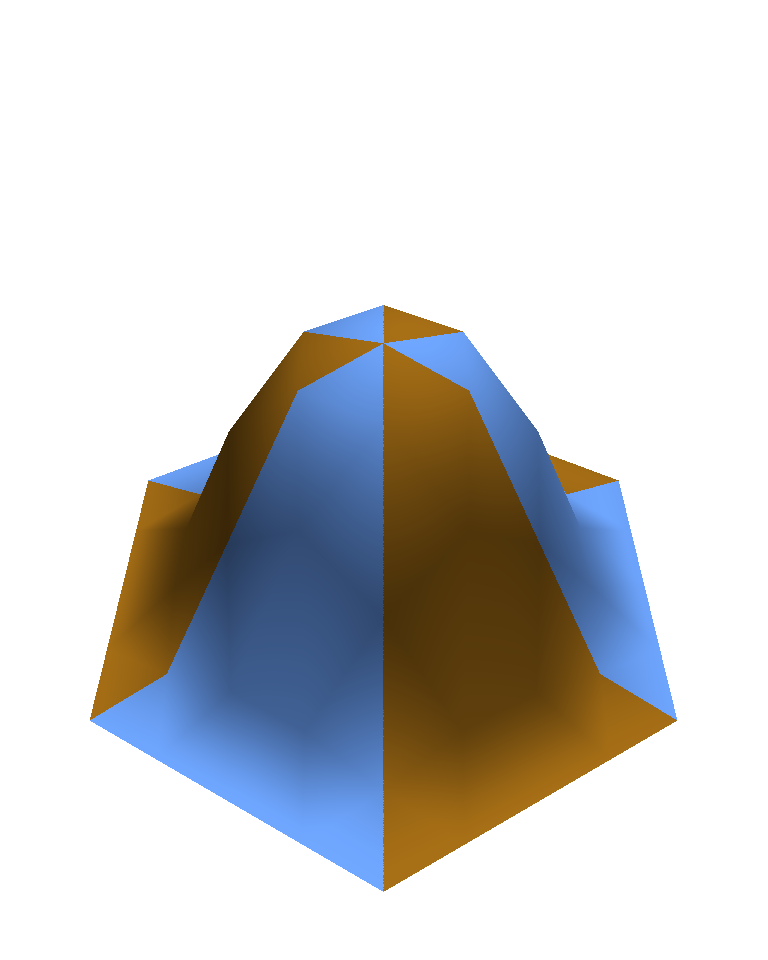}\quad
\includegraphics[scale=0.10, bb = 80 72 670 850, clip = True]{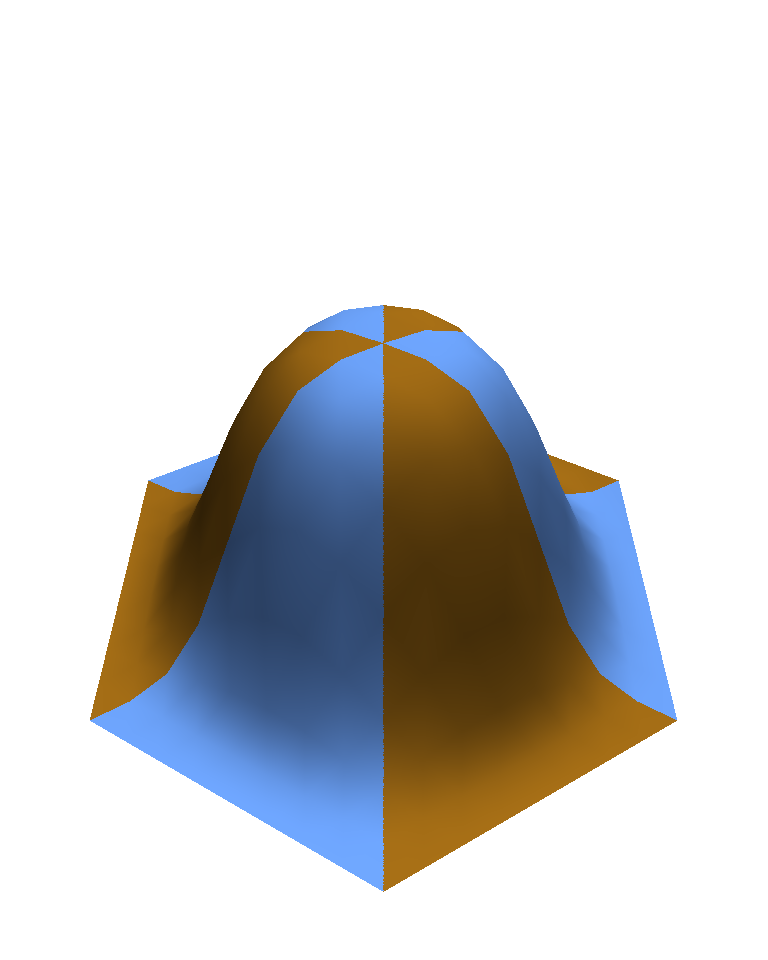}\quad
\includegraphics[scale=0.10, bb = 80 72 670 850, clip = True]{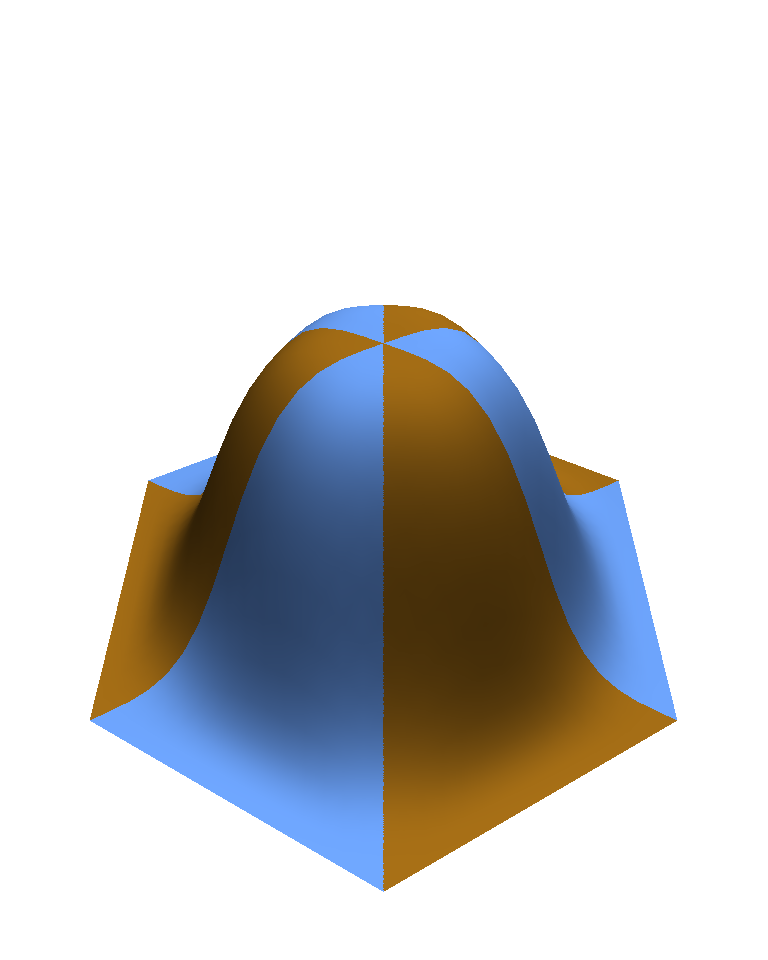}\quad
\includegraphics[scale=0.10, bb = 80 72 670 850, clip = True]{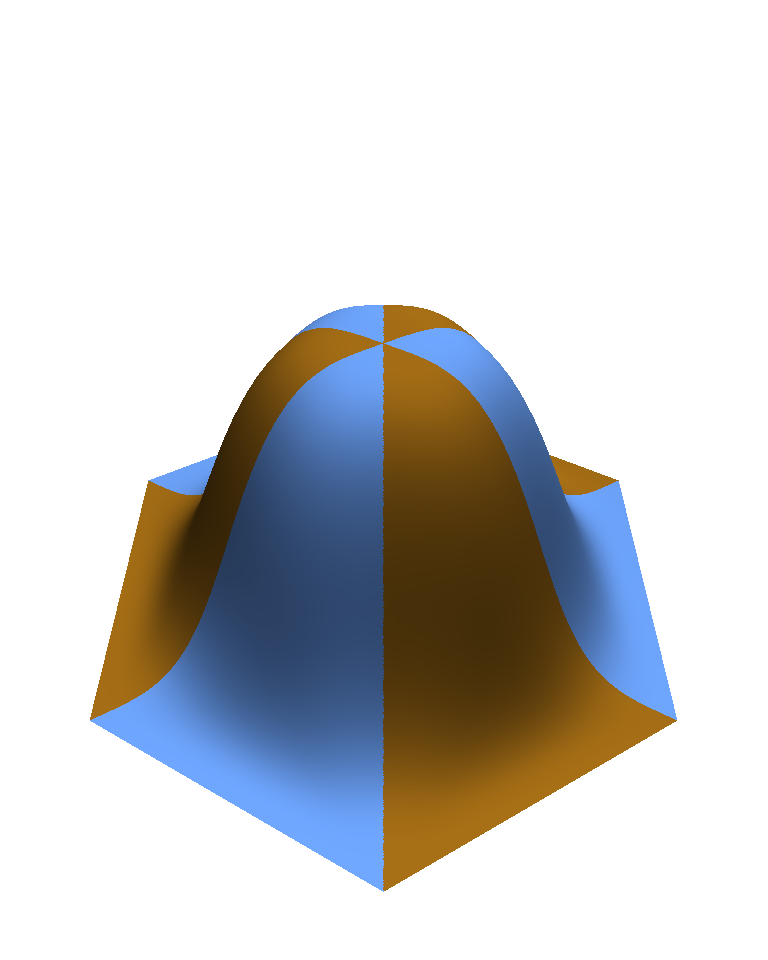}\\
\end{center}
\caption{The refined data after 1, 2, 3, 4, and 5 iterations, ray-traced using flat shading (top) and Phong shading (bottom) in Example \ref{ex:hexagon}.}\label{fig:flatPhong}
\end{figure}

\begin{example}\label{ex:hexagon}
Let $\vv_i := \big(\cos(2\pi i/6), \sin(2\pi i/6) \big)$, with $i = 1,2, \ldots, 6$, be the vertices of a regular hexagon centred at the origin $\vv_0 := (0, 0)$. Consider the triangulation consisting of the triangles $T^i := \langle \vv_0, \vv_i, \vv_{i+1}\rangle$, with $i = 1, 2, \ldots, 5$, and $T^6 := \langle \vv_0, \vv_6, \vv_{1}\rangle$. We consider the spline on this triangulation defined by zero initial data in \eqref{eq:InitialData12}, except for the value $1$ for the point evaluation at~$\vv_0$.

After the initialization step, the first 5 iterations of the subdivision step are shown in the top row of Figure \ref{fig:flatPhong}. However, in general we do not need this many iterations. Since the Hermite subdivision scheme also computes first derivatives, we know the surface normal at each vertex. Linearly interpolating these normals along each triangle and ray-tracing with a Phong shading model yields a seemingly flawless visualization after 2-3 iterations, while for flat shading 4-5 iterations are needed; see Figure \ref{fig:flatPhong} for a comparison.
\end{example}

\begin{example}\label{ex:derivatives}
Consider the triangulation $\PS$ with two triangles $T = \langle \vv_1, \vv_2, \vv_3\rangle$ and $\tilde{T} = \langle \tilde{\vv}_1, \vv_2, \vv_3\rangle$, with
\[ \vv_1 = (0,0),\quad \vv_2 = (1,0), \quad \vv_3 = (0,1),\quad \tilde{\vv}_1 = (1,1).\]
We consider the spline $f\in \SSS_{12}$ on this triangulation defined by zero initial data in \eqref{eq:InitialData12}, except for the value $1$ for the point evaluation at~$\vv_3$. Since the Hermite scheme computes derivatives up to order three, it is easy to plot the third derivatives of $f$ in Figure~\ref{fig:thirdderivatives}. Note that all third derivatives are continuous on $\PS$, with the exception of the third-order cross-boundary derivative $f^{\mm\mm\mm}$, as expected. 

Applying the scheme to data sampled from a random quintic polynomial on $\PS$ reproduced this polynomial exactly in exact arithmetic, and up to machine accuracy in floating point arithmetic.
\end{example}

\begin{figure}
\begin{center}
\includegraphics[scale=0.13, bb = 72 32 688 520, clip = True]{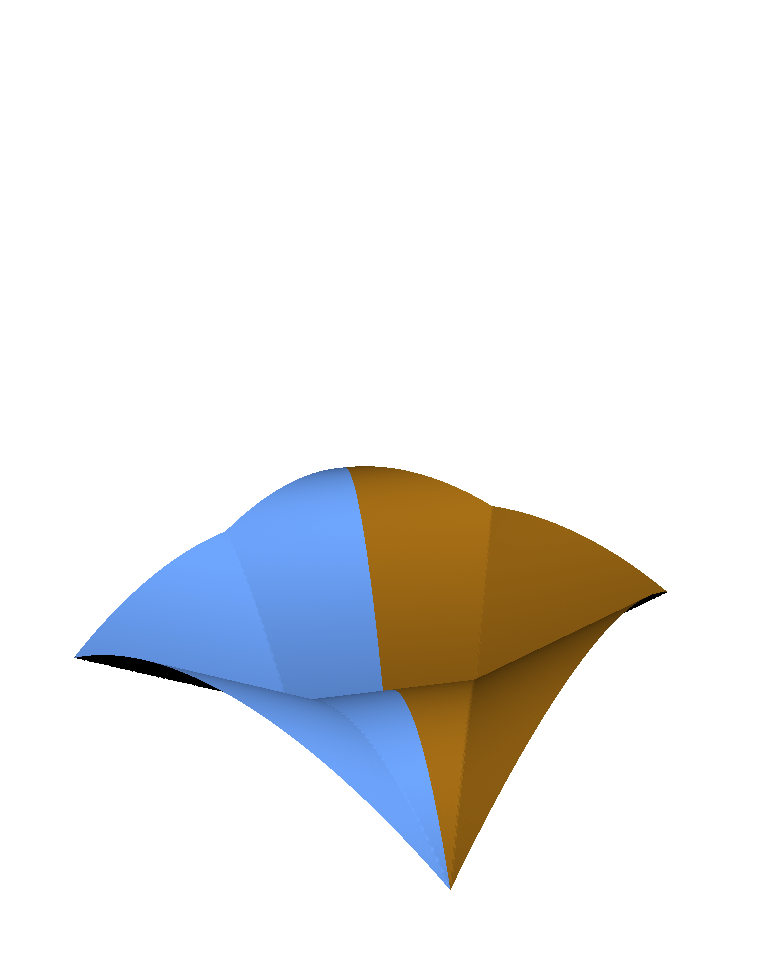}\quad
\includegraphics[scale=0.13, bb = 72 32 688 520, clip = True]{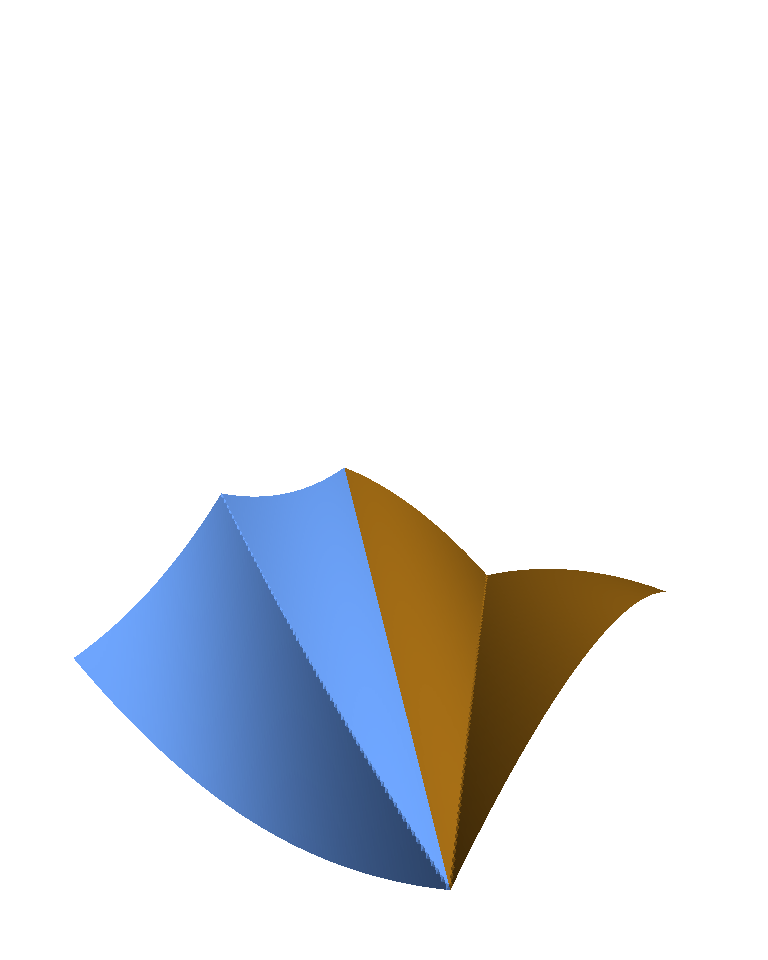}\quad
\includegraphics[scale=0.13, bb = 72 32 688 560, clip = True]{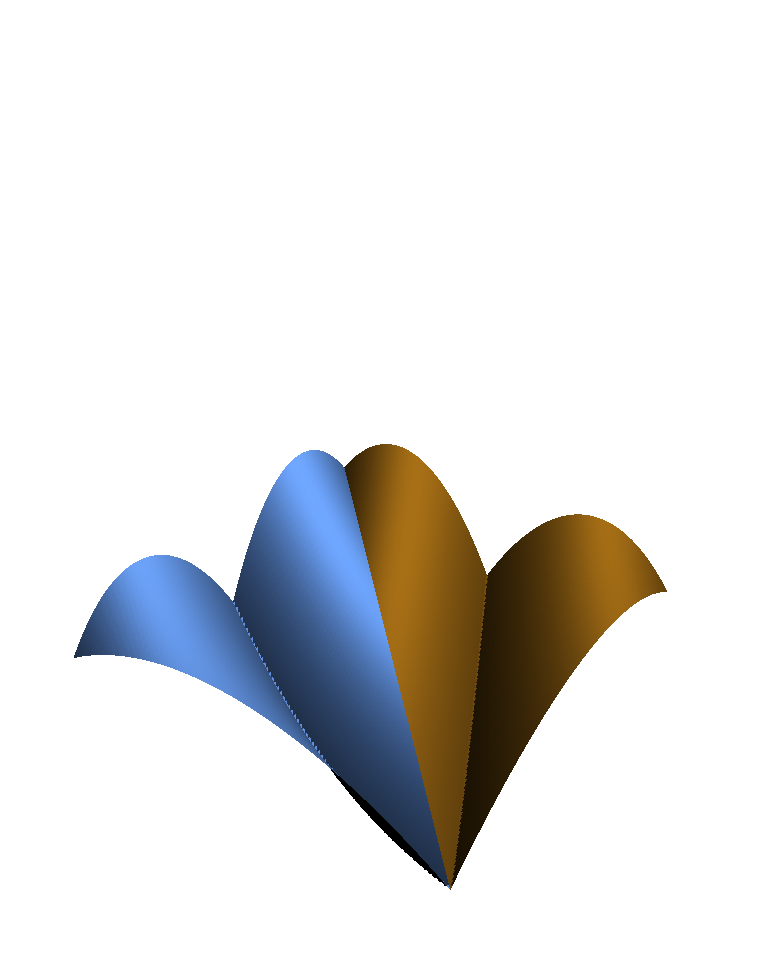}\quad
\includegraphics[scale=0.13, bb = 72 50 688 560, clip = True]{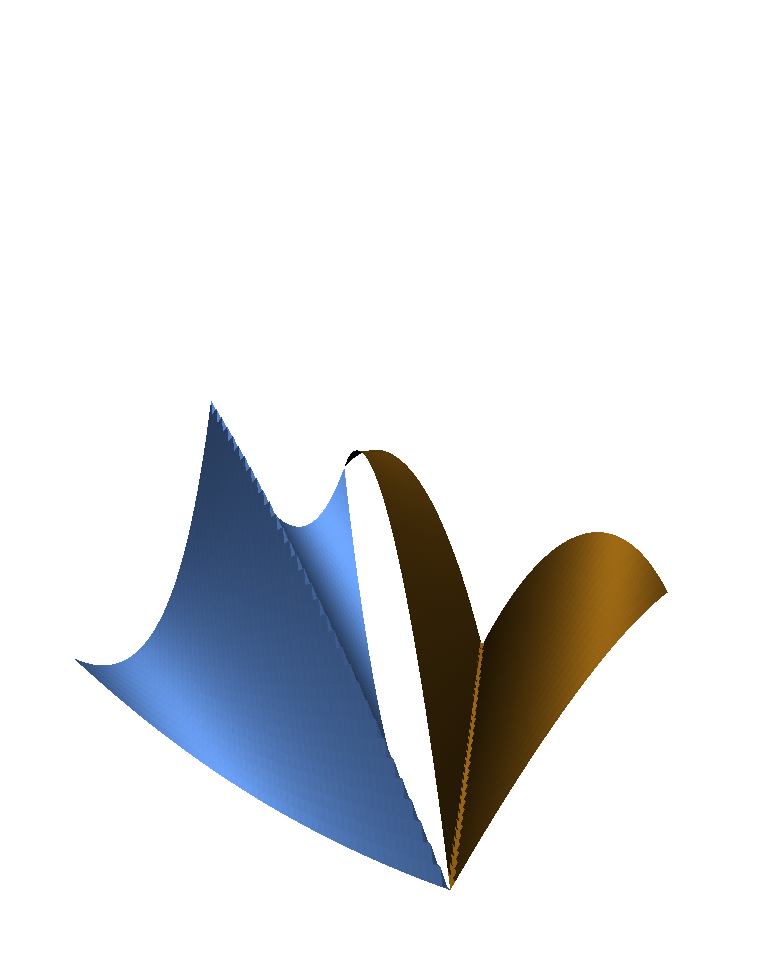}
\end{center}
\caption{From left to right, the third-order derivatives $f^{\tt\tt\tt}, f^{\tt\tt\mm}, f^{\tt\mm\mm}, f^{\mm\mm\mm}$ of the limit function $f$ in Example~\ref{ex:derivatives}.}\label{fig:thirdderivatives}
\end{figure}

\section{Conclusion and a final remark}
\noindent We have introduced a nodal macro-element on the 12-split for the space of quintic splines that are locally $C^3$ and globally $C^2$. For quickly evaluating any such spline, a Hermite subdivision scheme is derived, implemented, and tested in the computer algebra system \Sage. Using the available first derivatives for Phong shading, visually appealing plots can be generated after just a couple of refinements.

\begin{remark}
It would be natural to consider the macro-element introduced by Powell and Sabin and the macro-element in this paper as the first two entries in the following sequence. Let $\SSS^{2n-1, n}_{3n-1}$ be the space of piecewise polynomials of degree $3n-1$ with global $C^n$-smoothness and $C^{2n-1}$-smoothness within each macro-triangle. We define a set $\Lambda^n$ of functionals on $\SSS^{2n-1, n}_{3n-1}$ as follows. For every vertex $\vv$ of $\PS$, the set $\Lambda^n$ contains point evaluations at $\vv$ of the spline and of its partial derivatives up to order $2n-1$. The remaining elements of $\Lambda^n$ are, for each edge $e$ in $\PS$ and $k = 1,\ldots, n$, point evaluations of $k$ cross-boundary derivatives of order $k$ along $e$.

Although the cardinality of $\Lambda^n$ matches $\dim \SSS^{2n-1, n}_{3n-1}$, which can be shown to be $(15n^2 + 9n)/2$, it unfortunately does not in general form a basis for the dual to $\SSS^{2n-1, n}_{3n-1}$. The pattern breaks down first for $n = 3$, for which $\SSS^{5,3}_8$ has dimension $81$. While on each macro-triangle the 81 functionals in $\Lambda^3$ are linearly independent in the space of $C^4$-smooth octics, they are dependent in the space of $C^5$-smooth octics, with corank three.
\end{remark}

\section{Acknowledgments}
\noindent We wish to express our gratitude to the referees for their useful comments, which helped to improve the presentation of this paper. The second author was supported by a FRINATEK grant, project number 222335, from the Research Council of Norway.


\begin{bibdiv}
\begin{biblist}

\bib{Alfeld00}{article}{
   author={Alfeld, Peter},
   title={Bivariate spline spaces and minimal determining sets},
   note={Dedicated to Professor Larry L. Schumaker on the occasion of his
   60th birthday},
   journal={J. Comput. Appl. Math.},
   volume={119},
   date={2000},
   number={1-2},
   pages={13--27},
   issn={0377-0427},
}

\bib{Alfeld.Schumaker02}{article}{
   author={Alfeld, Peter},
   author={Schumaker, Larry L.},
   title={Smooth macro-elements based on Powell-Sabin triangle splits},
   journal={Adv. Comput. Math.},
   volume={16},
   date={2002},
   number={1},
   pages={29--46},
   issn={1019-7168},
}

\bib{Cohen.Lyche.Riesenfeld13}{article}{
   author={Cohen, Elaine},
   author={Lyche, Tom},
   author={Riesenfeld, Richard F.},
   title={A B-spline-like basis for the Powell-Sabin 12-split based on
   simplex splines},
   journal={Math. Comp.},
   volume={82},
   date={2013},
   number={283},
   pages={1667--1707},
   issn={0025-5718},
}

\bib{D.L.M.S:S00}{article}{
   author={D{\ae}hlen, Morten},
   author={Lyche, Tom},
   author={M{\o}rken, Knut},
   author={Schneider, Robert},
   author={Seidel, Hans-Peter},
   title={Multiresolution analysis over triangles, based on quadratic
   Hermite interpolation},
   note={Dedicated to Professor Larry L. Schumaker on the occasion of his
   60th birthday},
   journal={J. Comput. Appl. Math.},
   volume={119},
   date={2000},
   number={1-2},
   pages={97--114},
   issn={0377-0427},
}

\bib{DynLyche98}{article}{
   author={Dyn, Nira},
   author={Lyche, Tom},
   title={A Hermite subdivision scheme for the evaluation of the
   Powell-Sabin $12$-split element},
   conference={
      title={Approximation theory IX, Vol. 2},
      address={Nashville, TN},
      date={1998},
   },
   book={
      series={Innov. Appl. Math.},
      publisher={Vanderbilt Univ. Press},
      place={Nashville, TN},
   },
   date={1998},
   pages={33--38},
}

\bib{Davydov.Stevenson05}{article}{
   author={Davydov, Oleg},
   author={Stevenson, Rob},
   title={Hierarchical Riesz bases for $H^s(\Omega),\
   1<s<{5\over2}$},
   journal={Constr. Approx.},
   volume={22},
   date={2005},
   number={3},
   pages={365--394},
   issn={0176-4276},
}

\bib{Davydov.Yeo13}{article}{
   author={Davydov, Oleg},
   author={Yeo, Wee Ping},
   title={Refinable $C^2$ piecewise quintic polynomials on
   Powell-Sabin-12 triangulations},
   journal={J. Comput. Appl. Math.},
   volume={240},
   date={2013},
   pages={62--73},
   issn={0377-0427},
}

\bib{Goodman.Hardin06}{article}{
   author={Goodman, Tim},
   author={Hardin, Doug},
   title={Refinable multivariate spline functions},
   conference={
      title={Topics in multivariate approximation and interpolation},
   },
   book={
      series={Stud. Comput. Math.},
      volume={12},
      publisher={Elsevier B. V., Amsterdam},
   },
   date={2006},
   pages={55--83},
}

\bib{Hong.Schumaker04}{article}{
   author={Hong, Don},
   author={Schumaker, Larry L.},
   title={Surface compression using a space of $C^1$ cubic splines with a
   hierarchical basis},
   note={Geometric modelling},
   journal={Computing},
   volume={72},
   date={2004},
   number={1-2},
   pages={79--92},
   issn={0010-485X},
}

\bib{Jia.Liu08}{article}{
   author={Jia, Rong-Qing},
   author={Liu, Song-Tao},
   title={$C^1$ spline wavelets on triangulations},
   journal={Math. Comp.},
   volume={77},
   date={2008},
   number={261},
   pages={287--312 (electronic)},
   issn={0025-5718},
}

\bib{Lai.Schumaker03}{article}{
   author={Lai, Ming-Jun},
   author={Schumaker, Larry L.},
   title={Macro-elements and stable local bases for splines on Powell-Sabin
   triangulations},
   journal={Math. Comp.},
   volume={72},
   date={2003},
   number={241},
   pages={335--354},
   issn={0025-5718},
}
\bib{Lai.Schumaker07}{book}{
   author={Lai, Ming-Jun},
   author={Schumaker, Larry L.},
   title={Spline functions on triangulations},
   series={Encyclopedia of Mathematics and its Applications},
   volume={110},
   publisher={Cambridge University Press},
   place={Cambridge},
   date={2007},
   pages={xvi+592},
   isbn={978-0-521-87592-9},
   isbn={0-521-87592-7},
}

\bib{Maes.Bultheel06}{article}{
   author={Maes, Jan},
   author={Bultheel, Adhemar},
   title={$C^1$ hierarchical Riesz bases of Lagrange type on Powell-Sabin
   triangulations},
   journal={J. Comput. Appl. Math.},
   volume={196},
   date={2006},
   number={1},
   pages={1--19},
   issn={0377-0427},
}

\bib{WebsiteGeorg}{article}{
   author={Muntingh, Georg},
   title={Personal Website},
   eprint={https://sites.google.com/site/georgmuntingh/academics/software}
}

\bib{Oswald92}{article}{
   author={Oswald, Peter},
   title={Hierarchical conforming finite element methods for the biharmonic
   equation},
   journal={SIAM J. Numer. Anal.},
   volume={29},
   date={1992},
   number={6},
   pages={1610--1625},
   issn={0036-1429},
}

\bib{Powell.Sabin77}{article}{
   author={Powell, Michael J. D.},
   author={Sabin, Malcolm A.},
   title={Piecewise quadratic approximations on triangles},
   journal={ACM Trans. Math. Software},
   volume={3},
   date={1977},
   number={4},
   pages={316--325},
   issn={0098-3500},
}

\bib{Sage}{book}{
   author={Stein, William A.},
   author={others},
   organization = {The Sage Development Team},
   title = {{S}age {M}athematics {S}oftware ({V}ersion 4.7.2)},
   eprint = {{\tt http://www.sagemath.org}},
   date = {2012},
}

\bib{Schumaker.Sorokina06}{article}{
   author={Schumaker, Larry L.},
   author={Sorokina, Tatyana},
   title={Smooth macro-elements on Powell-Sabin-12 splits},
   journal={Math. Comp.},
   volume={75},
   date={2006},
   number={254},
   pages={711--726 (electronic)},
   issn={0025-5718},
}

\bib{Speleers13}{article}{
   author={Speleers, Hendrik},
   title={Construction of normalized B-splines for a family of smooth spline
   spaces over Powell-Sabin triangulations},
   journal={Constr. Approx.},
   volume={37},
   date={2013},
   number={1},
   pages={41--72},
   issn={0176-4276},
}
\end{biblist}
\end{bibdiv}
\end{document}